\documentclass[11pt,letterpaper]{amsart}
\usepackage[dvipsnames]{xcolor}
\usepackage[unicode]{hyperref}
\hypersetup{
	linktoc=all,     
	colorlinks = true,
	citecolor = Brown,
	anchorcolor = green,
	linkcolor = darkgray}
\usepackage[top=1.2in, bottom=1.2in, left=1.3in, right=1.25in]{geometry}
\numberwithin{equation}{section}
\usepackage{amsmath,amssymb,bm}
\usepackage{amsaddr}

\newtheorem{theorem}{Theorem}[section]
\newtheorem{lemma}[theorem]{Lemma}
\newtheorem{proposition}[theorem]{Proposition}

\theoremstyle{definition}
\newtheorem{definition}[theorem]{Definition}

\usepackage{indentfirst}
\usepackage{fancyhdr}
\usepackage{color}
\usepackage{mathrsfs}
\theoremstyle{remark}
\newtheorem{remark}[theorem]{Remark}

\newcommand{\pa}{\partial}
\allowdisplaybreaks[3]

\begin{document}

\title{Global well-posedness of non-integrable hyperbolic-elliptic Ishimori system in the critical Sobolev space}
\author[ ]{Zexian Zhang}
\address{ \vspace*{-20pt}Shanghai Center for Mathematical Sciences, Fudan University, Shanghai, China. \\
E-mail:  {\rm\textcolor[rgb]{0.00,0.00,0.84}{23110840019@m.fudan.edu.cn}}}
\author[ ]{ Yi Zhou$^{\ast}$ }
\address{\vspace*{-20pt}School of Mathematical Sciences, Fudan University, Shanghai, China.\\
E-mail: {\rm\textcolor[rgb]{0.00,0.00,0.84}{yizhou@fudan.edu.cn}}}
\thanks{$^{\ast}$Corresponding author}



\begin{abstract}
We consider the Cauchy problem for the hyperbolic-elliptic Ishimori system with general coupling constant $\kappa \in \mathbb{R}$ and prove global well-posedness in the critical Sobolev space. The proof relies primarily on new bilinear estimates, which are established via a novel div-curl lemma first introduced by the second author in \cite{zhou_1+2dimensional_2022}. Our approach combines the caloric gauge technique with $U^p$-$V^p$ type Strichartz estimates to handle the hyperbolic structure of the equation. The results extend previous work on the integrable case $(\kappa = 1)$ to general $\kappa$ and provide a unified framework which also works for the hyperbolic and elliptic Schrödinger maps in dimensions $d \ge 2$.
\end{abstract}

\maketitle

\tableofcontents

\section{Introduction} \label{sec1}	

We consider the hyperbolic-elliptic Ishimori system, a two-dimensional topological spin field model with the form
\begin{equation}\label{ishimori}
  \begin{cases} 
    \partial_t S = S \times (\partial_{x_1}^2 S - \partial_{x_2}^2 S) + \kappa (\partial_{x_1} \phi \cdot \partial_{x_1} S + \partial_{x_2} \phi \cdot \partial_{x_2} S), & \text{on } \mathbb{R}^2 \times \mathbb{R}, \\
    \Delta_x \phi = 2 S \cdot (\partial_{x_1} S \times \partial_{x_2} S), \\
    S|_{t=0} = S_0,
  \end{cases}
\end{equation}
where $S: \mathbb{R}^2 \times \mathbb{R} \to \mathbb{S}^2$ is a spin field taking values in the unit sphere, $\phi: \mathbb{R}^2 \times \mathbb{R} \to \mathbb{R}$ is a scalar potential, and $\kappa \in \mathbb{R}$ is the coupling constant. The system was introduced by Ishimori in \cite{ishimori_multivortex_1984} as a two-dimensional generalization of the two-dimensional Heisenberg equation in ferromagnetism. The potential $\phi$ is related to the topological charge density $2S \cdot (\partial_{x_1} S \times \partial_{x_2} S)$. The total topological charge, defined by
\begin{equation}
  Q  = \frac{1}{4\pi} \int_{\mathbb{R}^2} S \cdot (\partial_{x_1} S \times \partial_{x_2} S),
\end{equation}
represents the degree of the mapping $S: \mathscr{T} (\cong \mathbb{R}^2 \cup \{\infty\} )\to \mathbb{S}^2$.

Assuming that $S(t,x)$ is a solution to the Ishimori system \eqref{ishimori} with initial data $S_0$, the scaled function $S_\lambda(t,x) = S(\lambda^2 t,\lambda x)$ is also a solution with initial data $S_0(\lambda x)$. A direct scaling computation gives \begin{equation}
  \|S_\lambda(t)\|_{\dot{H}^{s}(\mathbb{R}^d)} = \lambda^{s -\frac{d}{2}}\|S(\lambda^2 t)\|_{\dot{H}^s(\mathbb{R}^d)},
\end{equation} 
so that in the two-dimensional case, the critical Sobolev space for $S$ is $\dot{H}^1(\mathbb{R}^2)$.


The local and global regularity properties of the Cauchy problem associated with the hyperbolic-elliptic Ishimori system have been extensively studied, see \cite{kenig_Cauchy_2005,soyeur_Cauchy_1992, wang_Local_2012}. Soyeur \cite{soyeur_Cauchy_1992} established local and global existence for small initial data in $H^3 $, and uniqueness for large data in $H^4 $. Wang \cite{wang_Local_2012} proved local well-posedness for small data in $H^\sigma_Q $  (see \eqref{1.3} below) for $\sigma > \frac{3}{2}$. A broader historical overview may be found in \cite[Sect. 9.2]{linares_Introduction_2015}.

Two special cases are of particular interest. When $\kappa = 1$, the system is completely integrable. In this setting, 
 Bejenaru, Ionescu, and Kenig \cite{bejenaru_stability_2011} established global well-posedness for small data in the critical Sobolev space $\dot{H}^1_Q(\mathbb{R}^2)$. Their analysis hinges on the observation that in the Coulomb gauge the magnetic terms vanish and the nonlinearities become non-derivative. When $\kappa = 0$, the system reduces to the two-dimensional hyperbolic Heisenberg equation, equivalent to the hyperbolic Schrödinger map. Of particular relevance to our approach is the seminal work of Bejenaru, Ionescu, Kenig, and Tataru \cite{bejenaru_global_2011}, who proved global existence for the elliptic Schrödinger map in critical Sobolev spaces for dimensions $d \ge 2$. Their proof relies crucially on the caloric gauge formulation and local smoothing estimates for the associated linear flow.

In this paper, we extend these results to the non-integrable Ishimori system with general $\kappa \in \mathbb{R}$ using the caloric gauge framework of \cite{bejenaru_global_2011}. This extension is highly non-trivial.
Unlike the integrable case $\kappa=1$, the system for general $\kappa$ is non-integrable and its nonlinearities are derivative-dependent, introducing significant analytic obstacles.
Moreover, compared to the elliptic Schr\"odinger map, the hyperbolic nature of the Ishimori equation prevents the direct use of local smoothing estimates, which were essential in \cite{bejenaru_global_2011}.
Instead, we employ $L^4_{t,x}$ Strichartz estimates in $U^p$-$V^p$ spaces (first introduced by Tataru in an unpublished work; see \cite{koch_Dispersive_2014} for a detailed introduction) and a new div-curl lemma first introduced by the second author in \cite{zhou_1+2dimensional_2022}.
This new div-curl lemma plays a crucial role in establishing bilinear estimates, especially for controlling the high-low and low-high frequency interactions in Bony's paraproduct decomposition, while high-high frequency interactions are handled by the $U^p$-$V^p$ type $L^4_{t,x}$ Strichartz estimates. We remark that $L^4_{t,x}$ Strichartz estimates have also played a role in the study of energy-critical Schr\"odinger maps \cite{2015A}, but the present analysis is distinguished by the hyperbolic nature of the Ishimori system and the reliance on the div-curl lemma rather than local smoothing estimates.
Our methods also work for both ultrahyperbolic and elliptic Schr\"odinger maps in dimensions $d \ge 2$, providing a unified approach.

Before stating our main results, we introduce some notation. For $\sigma \in [0,\infty)$, let $H^\sigma$ denote the usual Sobolev space of complex-valued functions. 
Given a point $Q \in \mathbb{S}^2$, we define the Sobolev space $H^{\sigma}_Q$ by
\begin{equation}\label{1.3}
H^{\sigma}_Q := \{f:\mathbb{R}^2\to  \mathbb{R}^3\ |\ |f|\equiv 1\ \text{a.e. and}\ f-Q \in H^{\sigma} \},
\end{equation}
which is equipped with the metric $d_{H^\sigma_Q}(f,g) = \|f-g\|_{H^{\sigma}}$. The spaces $\dot{H}^{\sigma}_Q$ are defined similarly. We define
\begin{equation}
  H_Q^{\infty} = \bigcap_{k=1}^{\infty}H^k_Q.
\end{equation}

Our main result in this paper is the following small-data global well-posedness result.
\begin{theorem}\label{thm1}
	Given $Q \in \mathbb{S}^2$, there exists $\varepsilon >0$ such that for any $S_0 \in H^{\infty}_Q$ with $\|S_0-Q\|_{\dot{H}^1}\le \varepsilon $, there exists a unique global solution $S    \in  C(\mathbb{R};H^{\infty}_Q)$ to the Ishimori system \eqref{ishimori} satisfying
	\begin{equation}\label{thm11}
	  \sup_t\|S(t)-Q\|_{\dot{H}^1}\lesssim \|S_0-Q\|_{\dot{H}^1},
	\end{equation}
	and for $k\in \mathbb{Z}_+,$
	\begin{equation}\label{thm12}
	  \sup_t\|S(t)\|_{ H_Q^k}\le C(k,\|S_0\|_{H^k_Q}).
	\end{equation}
Moreover, for any $\sigma \in [0,\sigma_1]$  the operator $T_Q: S_0\mapsto S(t)$ admits a continuous extension 
\begin{equation}
  T_Q: \mathcal{B}^{\sigma}_{\varepsilon(\sigma_1)}\to C(\mathbb{R},H_Q^{\sigma+1}),
\end{equation}
where 
\begin{equation}
   \mathcal{B}^{\sigma}_{\varepsilon(\sigma_1)} := \{f \in H^{\sigma+1}_Q \ |\ \|f-Q\|_{\dot{H}^1}\le  \varepsilon(\sigma_1)\}.
\end{equation}
\end{theorem}

\subsection{The modified spin model in caloric gauge}
Following the procedure in \cite{bejenaru_global_2011}, we construct the fields $\psi_m$ and the connection coefficients $A_m$, and derive the differentiated Ishimori equation satisfied by these functions. To fix the connection coefficients uniquely, we choose the \textbf{caloric gauge}, which is implemented by solving a heat equation and thus extending the spin field $S$ to include an auxiliary parabolic time variable $s \in [0,\infty)$.

Instead of working directly on the spin field $S$, we study its derivatives $\pa_\alpha S$ for $\alpha = t, 1,2$, which are tangent vectors in $T_{S(x,t)}\mathbb{S}^2$. Now suppose that there exists a  smooth frame $(v,w) = (v,S\times v) \in T_{S(x,t)}\mathbb{S}^2$. We introduce the complexified differentiated variables
\begin{equation}
  \psi_\alpha = v \cdot \pa_\alpha S + i w \cdot \pa_\alpha S,
\end{equation}
and the real connection coefficients 
\begin{equation}
  A_{\alpha} = w \cdot \pa_\alpha v.
\end{equation}

Since the vectors $(S,v, w )$ form an orthonormal frame for $T\mathbb{R}^{3}$, it follows that 
\begin{equation}\label{def}
	\begin{cases}
	 \pa_\alpha S = v \Re (\psi_\alpha) + w \Im (\psi_\alpha) ,\\
     \pa_\alpha v  = -S \Re (\psi_\alpha) +  w A_\alpha, \\
	   \pa_\alpha w = -S \Im (\psi_\alpha) - v A_\alpha.  \end{cases}
\end{equation}

Using the above formulas, one can verify that $\psi_\alpha$ and $A_\alpha$ satisfy the covariant curl relations
\begin{equation}\label{pcurl}
  (\pa_\alpha + iA_\alpha)\psi_\beta =  (\pa_\beta + iA_\beta)\psi_\alpha.
\end{equation}
Introducing the covariant derivative $D_\alpha = \partial_\alpha + i A_\alpha$, we can rewrite this as 
\begin{equation}
  D_\alpha \psi_\beta = D_\beta \psi_\alpha.
\end{equation}

Direct computation shows that 
\begin{align}
  &\pa_\alpha A_\beta - \pa_\beta A_\alpha = \Im(\psi_\alpha \overline{\psi_\beta})=: q_{\alpha\beta},\label{Acurl}\\ &[D_\alpha,D_\beta] = i q_{\alpha \beta}, 
\end{align}
where $q_{\alpha\beta}$ represents the curvature of the connection.

We now express the original Ishimori system \eqref{ishimori} in terms of $\psi_\alpha$ and $A_\alpha$.

For the Poisson equation in \eqref{ishimori}, we compute using \eqref{def}
\begin{align*}
	\Delta_x \phi& = 2S \cdot [(v \Re (\psi_1) + w \Im (\psi_1))\times (v \Re (\psi_2) + w \Im (\psi_2))]\\
	& = 2S \cdot (v\times w)(\Re (\psi_1)\Im (\psi_2) - \Re (\psi_2)\Im (\psi_1))\\
	& = 2\Im(\psi_2\overline{\psi_1})= 2(\pa_2 A_1 - \pa_1 A_2) = -2\epsilon_{ij}\pa_i A_j.
\end{align*}
Here and in what follows  we use the notation $\epsilon_{ij} = \delta_{1i}\delta_{2j}-\delta_{1j}\delta_{2i}$ and  the summation convention over repeated indices. Hence,
\begin{equation}
    \phi = 2(-\Delta)^{-\frac{1}{2}}\epsilon_{ij} R_i A_j, \qquad \partial_m \phi = 2\epsilon_{ij}R_m R_i A_j
\end{equation}
with $R_i = (-\Delta)^{-\frac{1}{2}}\partial_i$ denoting the Riesz transforms.

For the evolution equation, substituting \eqref{def} into \eqref{ishimori} gives
\begin{align*}
	\pa_t S &= S \times [v(\pa_1\Re \psi_1 - A_1\Im \psi_1) + w(\pa_1\Im \psi_1 + A_1\Re \psi_1)]\\
	&\ \ \ - S \times [v(\pa_2\Re \psi_2 - A_2\Im \psi_2) + w(\pa_2\Im \psi_2 + A_2\Re \psi_2)]\\
	&\ \ \ + \kappa[v \Re (\psi_1) + w \Im (\psi_1)]\pa_1 \phi+\kappa[v \Re (\psi_2) + w \Im (\psi_2)]\pa_2 \phi\\
	& = v (-\pa_1\Im \psi_1 - A_1\Re \psi_1 + \pa_2\Im \psi_2 + A_2\Re \psi_2 + \kappa\Re (\psi_1)\pa_1 \phi+\kappa\Re (\psi_2)\pa_2 \phi) \\
	&\ \ \ + w(\pa_1\Re \psi_1 - A_1\Im \psi_1 -\pa_2\Re \psi_2+  A_2\Im \psi_2 + \kappa\Im (\psi_1) \pa_1 \phi +\kappa\Im (\psi_2) \pa_2 \phi).
\end{align*}
Using $\psi_t = v \cdot \partial_t S + i w \cdot \partial_t S$, we obtain the compact expression
\begin{equation}\label{psit}
  \psi_t =  v \cdot \pa_t S + i w \cdot \pa_t S = i(D_1 \psi_1-D_2 \psi_2) + \kappa\psi_l \pa_l \phi,
\end{equation}
which expresses the time derivative of the spin field in terms of the spatial derivatives and connection coefficients.

Applying the compatibility condition \eqref{pcurl}, we derive the evolution equations for $\psi_m$ ($m=1,2$) as 
\begin{align}
	i D_t \psi_m & = i D_m \psi_t = - D_m(D_1 \psi_1-D_2 \psi_2) + i\kappa D_m(\psi_l \pa_l \phi)\notag\\
 & = -  D_m(D_1 \psi_1-D_2 \psi_2) + i\kappa (D_m \psi_l \pa_l \phi+ \psi_l \pa_l\pa_m\phi)\notag\\
 & = -(D_1^2 -D_2^2)\psi_m- i(q_{m 1}\psi_1 - q_{m 2}\psi_2) + 2i \epsilon_{ij}\kappa(D_m\psi_l R_l  R_i A_j + \psi_l \pa_m(R_l  R_i A_j)).
\end{align}
Expanding the covariant derivatives yields the following nonlinear hyperbolic Schrödinger equation
\begin{equation}\label{psi}
	\begin{aligned}
  	 (i\pa_t + \mu_l \pa_l^2 )\psi_m =&-2i \mu_l A_l \pa_l \psi_m +(A_t + \mu_l(A_l^2-i\pa_l A_l))\psi_m - i \mu_l\psi_l\Im (\psi_m\bar{\psi}_{l})  \\
	 & +  2i \epsilon_{ij}\kappa(D_m\psi_l R_l  R_i A_j + \psi_l \pa_m(R_l  R_i A_j)),
	\end{aligned}
\end{equation}
where $\mu_l = \delta_{1l} - \delta_{2l}$ reflects the hyperbolic signature.

Thus, the system \eqref{pcurl}, \eqref{Acurl}, \eqref{psit}, and \eqref{psi} constitutes the  modified spin model, a formulation of the Ishimori system in terms of the gauge-dependent variables $\psi_\alpha$ and $A_\alpha$.

To obtain a well-posed system, we impose Tao's \textbf{caloric gauge condition}, defined as follows:
\begin{definition}[Caloric gauge]\label{caloric}
	Let $S$ be a solution of \eqref{ishimori} in $C(\mathbb{R};H^{\infty}_Q)$ and $(v_\infty,w_\infty) = (v_\infty,Q\times v_\infty)$ be the orthonormal frame in $T_Q \mathbb{S}^2$. A \textit{caloric gauge} is  a tuple consisting of an extended map $\widetilde{S}: \mathbb{R}_+ \times \mathbb{R}\times \mathbb{R}^2\to \mathbb{S}^2$ and an orthonormal frame $(v,w) = (v,\widetilde{S}\times v)$ for $T_{\widetilde{S}}\mathbb{S}^2$  such that
	\begin{equation}\label{heat}
  \begin{cases}\pa_s \widetilde{S} = \Delta_x \widetilde{S}+  \widetilde{S}\sum_{m =1,2}|\pa_m \widetilde{S}|^2,\ \text{for}\ s\in[0,\infty)    \\ \widetilde{S}(0,t,x) = S(t,x),  \end{cases}
\end{equation}
and the following gauge conditions hold:
\begin{equation}
  A_s := w\cdot \pa_s v = 0,\ \lim_{s \to \infty}(v,w) = (v_\infty,w_\infty).
\end{equation}
\end{definition}

The existence and uniqueness of such a gauge for small initial data are guaranteed by Lemma \ref{heatgauge}.

We now define the extended variables $\psi_{\alpha'}, A_{\alpha'},\ \alpha' = s,t,1,2$ by
\begin{equation}
  \begin{cases} \psi_{\alpha'} = v\cdot \pa_{\alpha'}\widetilde{S} + i w\cdot \pa_{\alpha'}\widetilde{S},  \\ A_{\alpha'} = w\cdot \pa_{\alpha'} v  \end{cases}
\end{equation}
The parallel transport condition in the gauge definition implies the key gauge condition
\begin{equation}\label{As=0}
  A_s = 0.
\end{equation}

A similar computation shows that \eqref{heat} is equivalent to 
\begin{equation}
  \psi_s = D_1 \psi_1 + D_2 \psi_2.
\end{equation}

Taking the $D_\alpha$ derivative and using \eqref{pcurl} again, the heat equations for the extended variables $\psi_\alpha,\ \alpha=t, 1,2,$ take the form
\begin{equation}
  (\pa_s-\Delta_x)\psi_\alpha = 2i A_l \pa_l\psi_\alpha- (A_l^2 -i\pa_lA_l) \psi_\alpha + i \Im(\psi_\alpha \overline{\psi_l})\psi_l.
\end{equation}

Moreover, from $A_s = 0$ and \eqref{Acurl} we have
\begin{equation}
  \pa_s A_\alpha = \Im(\psi_s \overline{\psi_\alpha}) = \Im(\overline{\psi_\alpha}D_l \psi_l ),
\end{equation}
which, together with decay estimates  \eqref{gaugedecay} as $s \to \infty$, yields the integral representation
\begin{equation}\label{Aid}
  A_\alpha(s) = -\int_{s}^{+\infty}\Im(\overline{\psi_\alpha}D_l \psi_l )(r)\mathrm{d}r,\ \alpha = t,1,2.
\end{equation}

To establish continuous dependence on initial data, we also require the linearized Ishimori equation. Consider a one-parameter family of solutions $S^h$ to \eqref{ishimori} with $S^{h_0} = S$, and define $S_{\text{lin}} = \partial_h S^h|_{h_0}$. Linearizing \eqref{ishimori} gives
\begin{equation}	\pa_t S_{\text{lin}} = S_{\text{lin}}   \times  \mu_l \pa^2_l  S  + S   \times  \mu_l \pa^2_l  S_{\text{lin}}  +\kappa \pa_x \phi_{\text{lin}} \cdot \pa_x S +\kappa \pa_x \phi \cdot \pa_x S_{\text{lin}},
\end{equation}
with the constraint $S_{\text{lin}} \cdot S = 0$ (since variations preserve the sphere constraint).  The linearized potential $\phi_{\text{lin}}$ satisfies
\begin{equation}
  \begin{aligned}
  \Delta \phi_{\text{lin}}&  = 2[ S_{\text{lin}}\cdot (\partial_{1}S \times \pa_{ 2}S) +S\cdot (\partial_{ 1}S_{\text{lin}} \times \pa_{ 2}S) + S\cdot (\partial_{ _1}S \times \pa_{ 2}S_{\text{lin}}) ] \\
  & =2 S\cdot [\pa_1(S_{\text{lin}}\times \pa_2 S) -  \pa_2(S_{\text{lin}}\times \pa_1 S)]  \\
  & = 2\pa_1 [S\cdot (S_{\text{lin}}\times \pa_2 S)] - 2 \pa_2 [S\cdot (S_{\text{lin}}\times \pa_1 S)].
\end{aligned}
\end{equation}

Decomposing $S_{\text{lin}}$ in the frame as  
\begin{equation}
  S_{\text{lin}} = v \Re(\psi_{lin}) + w \Im(\psi_{lin}) ,
\end{equation}
and repeating the earlier computations yields the linearized equation for $\psi_{\text{lin}}$ 
\begin{equation}\label{linearized}
  \begin{aligned}
(i\pa_t + \mu_l \pa_l^2)\psi_{\text{lin}} = &-2i \mu_l A_l \pa_l \psi_{\text{lin}} +(A_t + \mu_l(A_l^2-i\pa_l A_l))\psi_{\text{lin}} - i \mu_l\psi_l\Im (\psi_{\text{lin}}\bar{\psi}_{l})\\
&~ + 2i \kappa \epsilon_{ij}R_l R_i A_j \cdot D_l \psi_{\text{lin}}-2i\kappa\psi_l R_l[R_1\Im(\overline{\psi_2} \psi_{\text{lin}}) - R_2\Im(\overline{\psi_1} \psi_{\text{lin}})].
  \end{aligned}
\end{equation}
\subsection{Outline of the proof}
We adopt the analytical framework introduced in \cite{bejenaru_global_2011} to establish our main results. Given a solution $S \in C(\mathbb{R}; H^\infty_Q)$ to the Ishimori system, our main goal is to prove an \textit{a priori} bound on
$\|S\|_{L^\infty_t (\dot{H}^{1}\cap H^{\sigma + 1}_Q)}$ for $\sigma$ in a fixed interval $ [0,\sigma_1]$. We shall use the homogeneous Littlewood-Paley decomposition and the notation of frequency envelopes.

\begin{definition}
	For  $k \in \mathbb{Z}$,  we define the standard  homogeneous Littlewood-Paley operator $P_k$. Let $\chi$ be a smooth cutoff to the region $[-1,1]$, and define $P_k$ and $P_{\le k}$ by
\begin{equation}
  \widehat{P_{\le k}f}(\xi) := \chi(2^{-k}|\xi|)\hat{f}(\xi),\  P_{k} = P_{\le k} - P_{\le\frac{k}{2}}.
\end{equation}
\end{definition}

\begin{definition}
	A positive sequence $\{b_k\}$ is a frequency envelope if it is $\ell^2$-bounded
	\begin{equation}
	  \sum_{k\in \mathbb{Z}}b_k^2 <\infty,
	\end{equation}
	and slowly varying, 
\begin{equation}
  b_{k}\le   2^{\delta|k-j|}b_j,\ k,j \in \mathbb{Z},
\end{equation}
where $\delta$ is a sufficiently small positive parameter.

An $\epsilon$-frequency envelope satisfies the additional condition
\begin{equation}
  \sum_{k}b_k^2 \le  \epsilon^2
\end{equation}
\end{definition}

Given an $\ell^2$ positive sequence $\alpha_k$, we often define its associated frequency envelope
\begin{equation}
  \alpha'_k  = \sup_{j}2^{-\delta|j-k|}\alpha_j.
\end{equation}

It is clear that $\alpha'_k$ is indeed a frequency envelope satisfying
\begin{equation}
  \alpha_k \le  \alpha'_k,\ \sum_{k} (\alpha'_k)^2 \lesssim \sum_{k} \alpha_k^2.
\end{equation}

Let  $S(s):= S(s,x,t)$ be the solution  to the heat flow in the caloric gauge with initial data $S   = S(x,t) \in C(\mathbb{R},H^{\infty}_Q)$. For $\sigma \in [0,\sigma_1+1]$, we introduce the frequency envelope associated to $S$:
\begin{align}
	&\gamma_k(\sigma) = \sup_{j\in \mathbb{Z}}2^{-\delta|k-j|}2^{(\sigma + 1) j}  \|P_{j}S(0)\|_{L_t^\infty L^2_x},\ \sigma \in  [0,\sigma_1 + 1].
\end{align} 
we also let $\gamma_k := \gamma_k(0)$. The existence of caloric gauge is ensured by the following lemma:
\begin{lemma} {\rm(\cite[Prop. 4.2]{bejenaru_global_2011})}\label{heatgauge}Given an arbitrary interval $I \subseteq \mathbb{R}$ and $S \in  C(I;H^{\infty}_Q)$ satisfying the smallness condition 
	\begin{equation}
	  \sum_{k \in \mathbb{Z}} 2^{2k}\|P_k S\|^2_{L^\infty_t L^2_x} \lesssim \sum_{k \in \mathbb{Z}}
	  \gamma_k^2 \ll 1,
	\end{equation} 
then  there exists a unique corresponding caloric gauge as defined
in Definition \ref{caloric}. Moreover, we have the bounds
\begin{equation}
  \|P_k(\widetilde{S},v,w)(s)\|_{L^\infty_t L^2_x}\lesssim \gamma_k(\sigma) \langle 2^{2k}s \rangle^{-20} 2^{-(\sigma + 1) k},\ \sigma \in [0,\sigma_1 ],
\end{equation}
and for $\sigma \in \mathbb{Z}_+$,
\begin{equation}\label{vwb}
  \sup_{k\in \mathbb{Z}}\sup_{s\in [0,\infty)}\langle s \rangle^{\frac{\sigma}{2}}2^{\sigma k}\| P_k   (\widetilde{S},v,w)(s)\|_{L^\infty_t L^2_x}<\infty.
\end{equation}
Consequently, we have 
\begin{equation}\label{gaugedecay}
	\begin{aligned}
	& \sup_{k \in \mathbb{Z}}\sup_{s \in [0,\infty)}\langle s \rangle^{\frac{\sigma}{2}}2^{(\sigma-1)k}\|P_k(\psi_m(s),A_m(s))\|_{L^\infty_t L^2_x}<\infty,\ \text{for}\ m=1,2, \\
	&\sup_{k \in \mathbb{Z}}\sup_{s \in [0,\infty)}\langle s \rangle^{\frac{\sigma}{2}}2^{ \sigma k}\|P_k(\psi_t(s),A_t(s))\|_{L^\infty_t L^2_x}<\infty.
	\end{aligned}
\end{equation}
\end{lemma}

We now state our main bootstrap arguments. Given an arbitrary interval $I\subseteq \mathbb{R}$ and a solution $S \in C(I;H^{\infty}_Q)$ satisfying the smallness condition in Lemma \ref{heatgauge},  we shall work with the caloric gauge and  the associated fields and connection coefficients $\psi_\alpha, A_\alpha$.  

For simplicity, we adopt the notation
\begin{equation}
  \bm{\psi}= (\psi_m,\overline{\psi_m})_{m=1,2},\ \bm{A} = (A_m)_{m=1,2}.
\end{equation}

Denote by $ G = L^\infty_t L^2_x \cap L^4_{t,x}$ the Strichartz space. In view of the bilinear estimate in Lemma \ref{bilinear}, we introduce the functional
\begin{equation}\label{Dcal}
 \mathcal{D}(u):= \|u\|_{G} + \sup_y\|u^{y}\cdot (i\pa_t + \mu_l \pa_l^2)u\|_{L^1_{t,x}}.
\end{equation}

We now introduce three families of frequency envelopes that will govern our bootstrap analysis. For $\sigma \in [0,\sigma_1]$ and $m=1,2$ we define
\begin{align}
	&a_k(\sigma) = \sup_{j\in \mathbb{Z}}2^{-\delta|k-j|}\sup_{s\ge 0}\langle 2^{2j}s \rangle^4 (2^{\sigma j}  \|P_{j}\bm{\psi}(s)\|_{G} + 1_{\{\sigma\ge \frac{1}{5}\}}2^{(\sigma-1)j} \|P_{j}\psi_t(s)\|_{G}),\label{afre}\\
	&b_k(\sigma) = \sup_{j\in \mathbb{Z}}2^{-\delta|k-j|} ( 2^{\sigma j}  \|P_{j}\bm{\psi}(0)\|_{G} + 1_{\{\sigma\ge \frac{1}{5}\}}2^{(\sigma-1)j} \|P_{j}\psi_t(0)\|_{G}),\label{bfre}\\
	& c_k(\sigma) = \sup_{j \in \mathbb{Z}}2^{-\delta|k-j|}2^{\sigma j}\|P_{j}\nabla S_0\|_{L^2_x}.\label{cfre}
\end{align}
Clearly,
\begin{equation}
  c_k(\sigma)\le b_k(\sigma) \le a_k(\sigma).
\end{equation}
We use the notation
\begin{equation}
  (a_k,b_k, c_k ):= (a_k(0),b_k(0), c_k(0) ) 
 \end{equation}
  to measure the critical regularity. Under the hypotheses of Theorem \ref{thm1}, $c_k$ forms an $\varepsilon$-frequency envelope.

Our proof strategy revolves around two key bootstrap propositions. The first controls the heat flow evolution.

\begin{proposition}[Heat flow bootstrap assumptions]\label{heatBA}
 Let $a_k(\sigma),b_k(\sigma)$ be defined as in \eqref{afre} and \eqref{bfre}. Suppose that  $b_k$ is a $\varepsilon^{\frac{3}{4}}$-frequency envelope and 
  \begin{align}
   &a_k(\sigma) \le  \varepsilon^{-\frac{1}{4}} b_k(\sigma),\ \mathcal{D}(P_k \bm{\psi} )\le  2^{-\sigma k}b_k(\sigma),\ \text{for}\ \sigma \in [0,\sigma_1],\\
     &\|P_k \psi_t\|_{G}\lesssim 2^k \varepsilon^{\frac{1}{2}}.
 \end{align}
Then we have the improved bounds
  \begin{align}
    a_k(\sigma)&\lesssim    b_k(\sigma),\ \text{for}\  \sigma \in [0,\sigma_1]\\
	\|P_k \psi_t(s)\|_G &\lesssim 2^k \langle 2^{2k }s\rangle^{-4}\varepsilon^{\frac{1}{2}},  \\
  \sup_y \|P_{k_1}\bm{\psi}(s) P_{k_2}w^y\|_{L^2_{t,x}}&\lesssim 2^{-\frac{|k_1-k_2|}{2}} 2^{-\sigma k} \langle 2^{2k}s \rangle^{-3} b_k(\sigma)\mathcal{D}(P_{k_2}w).
 \end{align}
\end{proposition}
 
The second bootstrap proposition concerns the hyperbolic Schrödinger evolution.

\begin{proposition}[Hyperbolic Schr\"odinger bootstrap assumptions]\label{SchrodingerBA}
	Let $b_k(\sigma),c_k(\sigma)$ be defined as in \eqref{bfre} and \eqref{cfre}. Suppose that
	\begin{align}
	  &\mathcal{D}(P_{k}\bm{\psi} )\le  b_{k}(\sigma) \le \varepsilon^{-\frac{1}{4}}c_{k}(\sigma), \ \text{for}\ \sigma \in [0,\sigma_1],\\
	  &\|P_k \psi_t \|_{G} \le  2^k \varepsilon^{\frac{1}{2}}.
	\end{align}
Then we have the improved bounds
	\begin{align}
	  &b_k(\sigma) +  \mathcal{D}(P_{k}\bm{\psi} )\lesssim c_{k}(\sigma), \ \text{for}\ \sigma \in [0,\sigma_1],\\
	  & \|P_k \psi_t \|_G \lesssim 2^k \varepsilon.
	\end{align}
\end{proposition}

Proposition \ref{heatBA} is proved in Section 3, while Proposition \ref{SchrodingerBA} is proved in Section 4. Following the strategy in \cite{bejenaru_global_2011}, we show that Proposition \ref{SchrodingerBA} implies Theorem \ref{thm1}. 
\begin{proof}[Proof of Theorem 1.1]

 Our starting point is the local-in-time existence and uniqueness for the Ishimori equation, due to Kenig and Nahmod \cite{kenig_Cauchy_2005}: if $S_0\in H_Q^\infty$, then
there exists  $T = T(\|S_0\|_{H^{2}_Q})$  and a unique solution $S\in C((-T,T),H^\infty_Q)$ of the Cauchy problem \eqref{ishimori}. Our goal is to use Proposition  \ref{SchrodingerBA} to prove the bounds 
\begin{align}
  &\|P_k \pa_x \phi\|_{L^\infty_t L^2_x}\lesssim 2^{-\sigma k}c_k(\sigma),\ \sigma \in [0,\sigma_1]. \label{thmb1}\\
  &\sum_{k}\|P_k(S-Q)\|_{L^\infty_t L^2_{x}}^2\lesssim \|S_0-Q\|^2_{L^2_x}.\label{thmb2}
\end{align}
Once these bounds are established, we can extend the local solution above to a unique global solution satisfying \eqref{thm11} and \eqref{thm12} in Theorem \ref{thm1} via a standard continuity argument.

\textbf{(a) Proof of \eqref{thmb1}.}  Define the quantity
\begin{equation}
  \Psi(T') = \sup_k \sup_{\sigma \in [0,\sigma_1]} c_k(\sigma)^{-1} 2^{\sigma k}\|P_k ( \pa_x \phi,\bm{\psi})\|_{L^\infty([-T',T'];L^2_x)}.
\end{equation}

We claim that 
\begin{equation}\label{PsiBA}
 \text{if}\ \Psi(T')\le  \varepsilon^{-\frac{1}{4}},\ \text{then}\ \Psi(T')\lesssim 1.
\end{equation}
By Proposition \ref{SchrodingerBA}, we have
\begin{equation}
  \|P_k \bm{\psi}\|_{L^\infty([-T',T'];L^2_x)}\lesssim 2^{- \sigma  k} c_k(\sigma),\ \sigma\in [0,\sigma_1].
\end{equation}
Using the relation $\pa_m \phi = \epsilon_{ij}R_m R_i A_j$ together with the bound for $\bm{A}$ at $s = 0$ obtained from the heat flow estimates of Section 3 (cf. \eqref{Ab} and \eqref{Aest}), we have
\begin{equation}
  \|P_k \pa_x \phi\|_{L^\infty_t L^2_x}\lesssim\|P_{k} \bm{A}(0)\|_{L^\infty_t L^2_x}\lesssim 2^{-\sigma k}c_k(\sigma).
\end{equation}

Using the local existence results and taking $\sigma_1 = 2$ in Proposition \ref{SchrodingerBA}, it is clear that there exists $T>0$ such that $\Psi(T)\lesssim 1$. A standard continuity argument based on \eqref{PsiBA} shows that $\Psi(T')\lesssim 1$ for arbitrarily large $T'$, completing the proof of \eqref{thmb1}.

\textbf{(b) Proof of \eqref{thmb2}.}  We consider the following bounds for the linearized equation.
\begin{proposition}
	Let $S$ be a solution to \eqref{ishimori} with small initial data $S_0$ satisfying $\|S_0\|_{\dot{H}^1}\ll 1 $. Under the caloric gauge related to $S$, for each initial datum $\psi_{\text{lin},0}\in H^\infty$ there exists a unique solution $\psi_{\text{lin}}\in C(\mathbb{R};H^\infty)$ to \eqref{linearized} satisfying 
	\begin{equation}\label{linearbound}
	  \sum_{k\in \mathbb{Z}}\|P_k \psi_{\rm{lin}}\|^2_{L^\infty_t L^2_x}\lesssim \|\psi_{\rm{lin},0}\|^2_{L^2_x}.
	\end{equation}
\end{proposition}
The proof of this result is identical to the proof of Proposition \ref{SchrodingerBA}, since the nonlinearities in \eqref{linearized} have a structure similar to that in \eqref{psi}.

We also need the following lemma from Tataru \cite{tataru_Rough_2005}.

\begin{lemma}{\rm(\cite[Prop. 3.13]{tataru_Rough_2005})}\label{join}
	Given $S_0^h \in H^\infty_Q$ with $\|S^h_0\|_{\dot{H}^1}\ll 1$ for $h = 0,1$, there exists a smooth one-parameter family $ S^h_0 \in C^\infty_h([0,1];H^\infty_Q)$ such that 
	\begin{align}
		& \|S^h_0\|_{\dot{H}^1}\ll 1,\ \forall h \in [0,1], \\
		& \int_{0}^1  \|\pa_h S^h_0\|_{L^2}\approx \|S^0_0 - S^1_0\|_{L^2_x}.
	\end{align}  
\end{lemma}

Let $(S_0^0,S_0^1) = (S_0,Q)$ in Lemma \ref{join}. Applying the bound \eqref{linearbound} to $\psi_{\text{lin}} = \pa_h S^h$ yields
\begin{equation}
  \sum_{k}\|P_k \pa_h S^h\|^2_{L^\infty_t L^2_x}\lesssim \|\pa_h S^h_0\|^2_{L^2_x}.
\end{equation}
Integrating in $h$ from $0$ to $1$ gives
\begin{equation}
   \sum_{k}\|P_k (S-Q)\|^2_{L^\infty_t L^2_x}\lesssim \|S_0-Q\|^2_{L^2_x},
\end{equation}
which is \eqref{thmb2}.

\textbf{(c) Proof of continuous extension.} It is sufficient to show that $T_Q$ admits a unique continuous extension 
\begin{equation}
  T_Q : H^1_Q \mapsto C(\mathbb{R};H_Q^1),
\end{equation} 
and the extension to higher Sobolev spaces follows similarly.

Consider a sequence of solutions $ S^n $ with initial data  $ S_0^n \in  H^\infty_Q$, and assume that  $(S_0^n)$ converges to $S_0$ in $H_Q^1$.  We shall prove that $(S^n)$ is Cauchy in $C(\mathbb{R};H_Q^1)$. Repeating the arguments used to prove \eqref{thmb2}, we can show that 
\begin{align}
\limsup_{n,m \to \infty}\|S^n-S^m\|_{L^\infty_t L^2_x}\lesssim  \limsup_{n,m \to \infty}\|S_0^n-S_0^m\|_{  L^2_x}=0.\label{thmpf}
\end{align}

Let $\{c_k^n\}$ be the frequency envelopes associated to $S_0^n$. By the convergence of $(S_0^n)$ in $H_Q^1$, these  envelopes also converge in $\ell^2$. Therefore 
\begin{equation}
  \lim_{k \to \infty}\sup_n \sum_{j>k}(c_j^n)^2 = 0.\label{thmpf2}
\end{equation}
Combining this with \eqref{thmb1} gives 
\begin{align}
	\|S^n-S^m\|_{L^\infty_t \dot{H}^1_x}&\lesssim  \|P_{\le  k}(S_0^n-S_0^m)\|_{ \dot{H}^1_x} +  \|P_{>  k}( S_0^n, S_0^m )\|_{ \dot{H}^1_x} \notag\\
	&\lesssim 2^k \|P_{\le  k}(S_0^n-S_0^m)\|_{L^2_x} + \big(\sup_n\sum_{j>k}(c_j^n)^2\big)^{\frac{1}{2}}.
\end{align}
Together with \eqref{thmpf} we have 
\begin{equation}
  \limsup_{n,m \to \infty}\|S^n-S^m\|_{L^\infty_t \dot{H}^1_x}\lesssim \big(\sup_n\sum_{j>k}(c_j^n)^2\big)^{\frac{1}{2}}.
\end{equation}
Letting $k \to \infty$ and using \eqref{thmpf2}, we obtain\begin{equation}
   \limsup_{n,m \to \infty}\|S^n-S^m\|_{L^\infty_t \dot{H}^1_x}=0,
\end{equation}
which completes the proof since $H^1_Q = \dot{H}^1_x \cap L_Q^2$.
\end{proof}

\section{Preliminaries} 

\subsection{Multilinear expression}
Following Tao \cite{tao_Global_2001}, we introduce a convenient notation for describing multi-linear expressions of product type.

Denote by $u^{y}(x) := u(x-y)$ the translation of $u(x)$. For multilinear operators, let $L$ be the integral form
\begin{equation}
  L(u_1,u_2, \cdots ,u_k)(x)  = \int K(y_1, \cdots ,y_k)u_1^{y_1}(x)\cdots u_k^{y_k}(x)\mathrm{d}y,
\end{equation} 
where $K$ can be an integrable kernel or, more generally, a bounded measure (including, for example, product-type expressions). The kernels may change from line to line, but we require that these kernels have uniformly bounded mass.

This $L$ notation will turn out to be useful for expressing matrix coefficients, Littlewood-Paley multipliers $P_k$ etc., whenever these structures are not being exploited. For example, the $L$ notation is invariant under permutation of standard Littlewood-Paley operators
\begin{align*}
 &L(P_k u_1,u_2, \cdots ,u_k) =L( u_1,u_2, \cdots ,u_k),\\
 &P_k L( u_1,u_2, \cdots ,u_k) =L( u_1,u_2, \cdots ,u_k).
\end{align*} 
The same holds for $P_{>k},\ P_{<k}$, etc. Furthermore, this notation also interacts well with the composition of  Littlewood-Paley operators and Riesz-type operators
\begin{align}
 L(R_m P_k u_1,u_2, \cdots ,u_k) =L(P_k u_1,u_2, \cdots ,u_k).
\end{align}

Multilinear estimates are not invariant under separate translations for the factors. To obtain similar bounds for  these $L$ notations, we shall allow translations in the multilinear estimates. For example, if we have the bound with a translation-invariant norm $\|\cdot\|_X$
$$\sup_{y_\alpha}\| u_1 u_2^{y_2} u_3^{y_3} \|_{X}\le C,$$
then for any trilinear form $L$ with integrable kernel we have 
\begin{align*}
  \|L(u_1,u_2,u_3)\|_{X} &\le\sup_{y_\alpha} \int K(y_1,y_2,y_3)\|u_1^{y_1} u_2^{y_2}u_3^{y_3}\|_{X}\mathrm{d}y \\
  &=\sup_{y_\alpha} \int K(y_1,y_2,y_3)\|u_1 u_2^{y_2}u_3^{y_3}\|_{X}\mathrm{d}y \lesssim C.
\end{align*}

\subsection{Linear estimates}
Let $G = L^4_{t,x} \cap L^\infty_t L^2_x$ denote the standard Strichartz space. To establish Strichartz estimates for solutions to the linear hyperbolic Schrödinger equation
\begin{equation}\label{linear}
    (i\partial_t + \mu_l \partial_l^2)u = N, \qquad \mu_l = \delta_{1l} - \delta_{2l},
\end{equation}
we follow the approach in \cite{ifrim_global_2025} and employ adapted $U^p$ and $V^p$ spaces. These spaces provide a flexible framework for controlling the dispersive properties of the solution operator $e^{it(\partial_1^2 - \partial_2^2)}$.

\begin{definition}
	Let $1\le p <\infty$. Then $U^p_{\text{UH}}$ is an atomic space whose atoms are piecewise solutions to the linear equation, i.e.
	\begin{equation}
	  u = \sum_{k}1_{[t_k,t_{k+1})}e^{it(\pa_1^2 - \pa_2^2)}u_k,\ \sum_{k}\|u_k\|_{L^2}^p = 1.
	\end{equation}

    And we equip the $U^p_{\text{UH}}$ with the norm
    \begin{equation}
      \|u\|_{U^p_{\text{UH}}} = \inf\{\sum_{\lambda}|c_\lambda|\ |\ u = \sum_{\lambda}c_\lambda u_\lambda,\ u_\lambda\ \text{are}\ U^p_{\text{UH}}\  \text{atoms}\}.
    \end{equation}

The space $V^p_{\text{UH}}$ consists of the right-continuous functions $v\in L^\infty_t L_x^2$ such that
     \begin{equation}
   \|v\|^p_{V^p_{\text{UH}}} = \|v\|^p_{L^\infty_t L_x^2} + \sup_{\{t_k\}\nearrow}\sum_k\|e^{-it_k(\pa_1^2-\pa_2^2)}v(t_k)-e^{-it_{k+1}(\pa_1^2-\pa_2^2)}v(t_{k+1}) \|_{L^2}^p,
    \end{equation}
where the supremum is taken over increasing sequences $\{t_k\}$.
\end{definition}

\begin{theorem}
	We have the following embeddings 
	\begin{equation}
	  U^{p}_{\text{UH}}\hookrightarrow V^{p}_{\text{UH}}\hookrightarrow U^{q}_{\text{UH}}\hookrightarrow L^\infty_t L_x^2,\ 1\le p <q <\infty.
	\end{equation}

	Let $DV^p_{\text{UH}}$ be the space of functions
	\begin{equation}
	  DV^p_{\text{UH}} = \{(i\pa_t + \pa_1^2-\pa_2^2)u\ | \ u \in V^p_{\text{UH}}\}
	\end{equation}
with the induced norm. For a solution $u$ to \eqref{linear} we have the elementary estimate:
\begin{equation}
  \|u\|_{V^p_{\text{UH}}}\lesssim \|u\|_{L^\infty_t L_x^2} + \|N\|_{DV^p_{\text{UH}}}.
\end{equation}
Moreover, we have the duality relation
\begin{equation}
  ( DV^p_{\text{UH}} )^* = U^{p'}_{\text{UH}},\ \ \text{for}\ \frac{1}{p} + \frac{1}{p'}=1.
\end{equation}
    Finally, for $G = L^4_{t,x} \cap L^\infty_t L^2_x$, we have the Strichartz-type estimate
    \begin{equation}\label{Stri}
        \|u\|_{G} \lesssim \|u\|_{L^\infty_t L^2_x} + \|N\|_{DV^2_{\text{UH}}} 
        \lesssim \|u|_{t=0}\|_{L^2_x} + \|N\|_{DV^2_{\text{UH}}} + I(u,N)^{\frac{1}{2}},
    \end{equation}
    where $I(u,N)$ denotes the interaction term
    \begin{equation}
        I(u,N) = \sup_y\|u^{y}\cdot N\|_{L^1_{t,x}},
    \end{equation}
    which arises naturally in energy estimates.
\end{theorem}

\begin{remark}
    The spaces $U^p_{\text{UH}}$ and $V^p_{\text{UH}}$ are adaptations of the standard $U^p$ and $V^p$ spaces to the specific dispersive properties of the hyperbolic Schrödinger operator $i\partial_t + \partial_1^2 - \partial_2^2$. Their construction follows the general theory developed in \cite[Chap. 4]{koch_Dispersive_2014}, where analogous spaces are introduced for various dispersive equations. The key properties—embeddings, duality, and connection to Strichartz estimates—are proved using the methods outlined therein.
\end{remark}

\subsection{Bilinear estimates}

The following \textbf{div-curl} lemma, which was first introduced by the second author \cite{zhou_1+2dimensional_2022}, plays a crucial role in our proof.
\begin{lemma} [div-curl Lemma] \label{lemdiv-curl}
	Suppose that $f^{ij}, i, j=1, 2$ satisfy 	
	$$
	\begin{gathered}
		\left\{\begin{array}{l}
			\partial_t f^{11}+\partial_x f^{12}=G^1, \\
			\partial_t f^{21}-\partial_x f^{22}=G^2,\\
		\end{array}\right. \\
		f^{11}, f^{12}, f^{21}, f^{22} \rightarrow 0, \text{ as  }x \rightarrow \infty,
	\end{gathered}
	$$
	then it holds that	
	\begin{equation}\label{divcurl}
	\begin{aligned}
		&\int_{-\infty}^{+\infty} \int_{-\infty}^{+\infty} f^{11} f^{22}+f^{12} f^{21}  \mathrm{ d} x \mathrm{ d} t \\
        \leq&~ 2( \|f^{11}  \|_{L^\infty_t L^1_x} + \|G^1\|_{L^1_{t,x}})\cdot(  \|f^{21} \|_{L^\infty_t L^1_x}+ \|G^2\|_{L^1_{t,x}}).
	\end{aligned}
	\end{equation}
	provided that the right side is bounded.
\end{lemma}

\begin{proof}
	 The same computation as in \cite{wang_periodic_2024} yields
	\begin{align*}
	&	\frac{\partial}{\partial t} \int_{x<y} f^{11} (t,x) f^{21} (t,y) \mathrm{ d} x\mathrm{ d} y+\int_{-\infty}^{+\infty}\left (f^{11} f^{22}+f^{12} f^{21}\right) \mathrm{ d} x \\
	=& \int_{-\infty}^{+\infty} \left( \int_{-\infty}^x f^{11} (t,y) \mathrm{ d} y\right)G^2(t,x) \mathrm{ d} x  +\int_{-\infty}^{+\infty} \left( \int^{+\infty}_x f^{21} (t,y) \mathrm{ d} y\right)G^1(t,x) \mathrm{ d} x ,
	\end{align*}
    then \eqref{divcurl} follows by integrating the above inequality with respect to $t$ over $\mathbb{R}$.
\end{proof} 

To proceed further we compute the conservation laws of \eqref{linear}.

Multiplying \eqref{linear} with $\bar{u}$ and taking the imaginary part, we have the mass conservation law
\begin{equation}
  \begin{aligned}
  \Im[\bar{u}(i\pa_t + \mu_l \pa_l^2)u] = \frac{1}{2}\pa_t |u|^2 + \mu_l \pa_l \Im (\bar{u}\pa_l u) = \Im (\bar{u} N).\\
  \end{aligned}
\end{equation}

Multiplying \eqref{linear} with $\pa_m\bar{u}$ and taking the real part, we have 
\begin{align}
	\Re[\pa_m \bar{u}(i\pa_t + \mu_l \pa_l^2)u] &= - \Im(\pa_m \bar{u} \pa_t u) + \mu_{l}\Re(\pa_m \bar{u} \pa_l^2 u)\notag\\
	& =  \frac{1}{2}  \Im(\pa_m u \pa_t \bar{u} - \pa_m \bar{u} \pa_t u) + \mu_{l}\pa_l \Re(\pa_m \bar{u} \pa_l u)-\mu_{l} \Re(\pa_m \pa_l\bar{u} \pa_l u) \notag\\
	& = \frac{1}{2}\pa_t \Im(\bar{u} \pa_{m}u)- \frac{1}{2}\pa_m\Im (\bar{u} \pa_t u )+ \mu_{l}\pa_l \Re(\pa_m \bar{u} \pa_l u)-\frac{1}{2}\mu_{l} \pa_m  |\pa_l u|^2 = \Re  (\pa_m\bar{u} N),
\end{align}
and 
\begin{align}
	\pa_m\Im (\bar{u} \pa_t u ) &=  \mu_l\pa_m\Im(i \bar{u} \pa_l^2 u) -\pa_m\Im (i\bar{u} N) =\mu_l\pa_m\Re(  \bar{u} \pa_l^2 u) -\pa_m\Re ( \bar{u} N)\notag\\
	& = \mu_l  \pa_m \pa_l \Re(\bar{u}\pa_l u) -  \mu_l  \pa_m |\pa_l  u|^2 - \pa_m \Re(\bar{u} N)\notag\\
	& = \frac{1}{2}\mu_l  \pa_m \pa^2_l|u|^2 -\mu_l  \pa_m |\pa_l  u|^2 - \pa_m \Re(\bar{u} N).
\end{align}
Combining the above, we obtain the momentum conservation law
\begin{equation}
  \frac{1}{2}\pa_t \Im(\bar{u} \pa_{m}u) + \mu_{l}\pa_l \Re(\pa_m \bar{u} \pa_l u) - \frac{1}{4}\mu_l  \pa_m \pa^2_l|u|^2 = \Re(\pa_m\bar{u}  N)-\frac{1}{2}\pa_m\Re(\bar{u} N).
\end{equation}

\begin{proposition}\label{bilinear1}
	Given solutions $u,v$ to the equations
	\begin{align}
	   &(i\pa_t + \mu_l\pa_l^2)u = N,\label{lin1}\\
	   & (i\pa_t + \mu_l\pa_l^2)v = N',\label{lin2}
	\end{align}
	then for $k_2\ge  k_1 + 80$  we have 
	\begin{equation}
		\begin{aligned}
		 &\sup_y \Big\|\|P_{k_1} u^y\|_{L^2_{\widehat{x}_m}} \|\Xi_m(D)P_{k_2} v\|_{L^2_{\widehat{x}_m}}\Big\|_{L^2_{t,x_m}}\\
		 &\lesssim 2^{-\frac{k_2}{2}} (\|P_{k_1} u\|_{L^\infty_t L^2_x}+ I(P_{k_1}u,P_{k_1}N)^{\frac{1}{2}})(\|P_{k_2} v\|_{L^\infty_t L^2_x}+ I(P_{k_2}v,P_{k_2}N')^{\frac{1}{2}}), \\
		\end{aligned}
	\end{equation}
where we denote $\widehat{x}_{m} = (x_l)_{l \ne m},\ I(u,N) := \sup_{y}\|u^{y} N\|_{L^1_{t,x}}$ and  $\Xi_m(D)$ is the zero-th order Fourier multiplier supported in the Fourier region $\{|\xi|\lesssim  |\xi_m |\}$ such that $ \sum_{m=1,2}\Xi_m(D) = \text{Id}$.

\end{proposition}
\begin{proof}
	Applying the operator $P_{k_1}$ and $\Xi_m(D)P_{k_2}$ to \eqref{lin1} and \eqref{lin2} respectively, we get the integrated mass conservation law of $ u^y_1 := P_{k_1} u^{y}$ and momentum conservation law of $v_2:= \Xi_m(D)P_{k_2} v$ as follows:
	\begin{align}
		& \frac{1}{2}\pa_t\int |u^y_1|^2 \mathrm{d}\widehat{x}_m + \mu_m \pa_m \int\Im (\overline{u^y_1}\pa_m u^y_1)\mathrm{d}\widehat{x}_m  =\int \Im (\overline{u^y_1} P_{k_1}N^{y})\mathrm{d}\widehat{x}_m,  \\
		&   \frac{1}{2}\pa_t\int  \Im(\overline{v_2} \pa_{m}v_2) \mathrm{d}\widehat{x}_m + \mu_{m}\pa_m \int\big( |\pa_m  v_2 |^2   - \frac{1}{4}  \pa_m^2 |v_2|^2 \big) \mathrm{d}\widehat{x}_m \\
		&\quad \quad   =\int   \Re(\pa_m\overline{v_2} \Xi_m(D)P_{k_2} N')-\frac{1}{2}\pa_m\Re(\overline{v_2} \Xi_m(D)P_{k_2} N')\mathrm{d}\widehat{x}_m  = 2^{k_2}\int L(P_{k_2}v, P_{k_2}N')\mathrm{d}\widehat{x}_m  .\notag
	\end{align}

	Applying the div-curl Lemma \ref{lemdiv-curl} to the above conservation laws with respect to the variable $(t,x_m)$, after integration by parts we have 
		\begin{align}
		&\sup_y\iint\Big(\int |u^y_1|^2 \mathrm{d}\widehat{x}_m \int  |\pa_m  v_2 |^2\mathrm{d}\widehat{x}_m \Big) \mathrm{d}t\mathrm{d}x_m   \notag \\
        &+\sup_{y}\iint \Big( \int \Re(\overline{u^y_1} \pa_{m}u^y_1) \mathrm{d}\widehat{x}_m\int \Re(\overline{v_2} \pa_{m}v_2) - \int \Im(\overline{u^y_1} \pa_{m}u^y_1) \mathrm{d}\widehat{x}_m\int \Im(\overline{v_2} \pa_{m}v_2) \mathrm{d}\widehat{x}_m \Big) \mathrm{d}t\mathrm{d}x_m  \notag\\
		&=: A + B\lesssim 2^{k_2} (\|P_{k_1} u\|^2_{L^\infty_t L^2_x}+ I(P_{k_1}u,P_{k_1}N))(\|P_{k_2} v\|^2_{L^\infty_t L^2_x}+ I(P_{k_2}v,P_{k_2}N')). \notag
		\end{align} 

It is clear that 
\begin{align}
  &\sup_{y}\Big\|\|P_{k_1} u^{y}\|_{L^2_{\widehat{x}_m}} \|\Xi_m(D)P_{k_2} v\|_{L^2_{\widehat{x}_m}}\Big\|_{L^2_{t,x_m}}^2\notag\\
  &\lesssim 2^{-2k_2}\sup_{y'}\Big\|\|u_1^{y'}\|_{L^2_{\widehat{x}_m}} \|\pa_m v_2\|_{L^2_{\widehat{x}_m}}\Big\|^2_{L^2_{t,x_m}} \lesssim 2^{-2k_2}A, 
\end{align}
and 
\begin{align}
	|B| &\lesssim    \sup_y \iint\Big( \|u_1^y\|_{L^2_{\widehat{x}_m}}\|\pa_m u^y_1\|_{L^2_{\widehat{x}_m}}\|v_2\|_{L^2_{\widehat{x}_m}}\|\pa_m v_2\|_{L^2_{\widehat{x}_m}}\Big)\   \mathrm{d}t\mathrm{d}x_m \notag \\
	&\lesssim 2^{k_1-k_2} \sup_{y} \iint\Big( \|u_1^y\|^2_{L^2_{\widehat{x}_m}} \|\pa_m v_2\|^2_{L^2_{\widehat{x}_m}}\Big)\   \mathrm{d}t\mathrm{d}x_m \lesssim 2^{k_1-k_2}A.
\end{align}

The conclusion follows from plugging the above inequalities and using the fact that $k_2\gg k_1$. 
\end{proof}

Using Proposition \ref{bilinear1} and the Strichartz estimate, we obtain the core bilinear estimate in this paper:
\begin{proposition}\label{bilinear2} We have 
	\begin{equation}\label{bilinear}
		\begin{aligned}
		  &\sup_y\|P_{k_1}u^y P_{k_2} v\|_{L^2_{t,x}} \lesssim 2^{-\frac{|k_1-k_2|}{2}}\mathcal{D}(P_{k_1}u )\mathcal{D}(P_{k_2}v),
		\end{aligned}
	\end{equation} 
where  $ \mathcal{D}(u):= \|u\|_{G} + \sup_y\|u^{y}\cdot (i\pa_t + \mu_l \pa_l^2)u\|_{L^1_{t,x}}$.
\end{proposition}
\begin{proof}
	If $|k_1-k_2|\le  100$, the bound for \eqref{bilinear} is straightforward. If $|k_1-k_2|\ge  100$, without loss of generality we assume $k_2\ge k_1$. We expand 
	\begin{align}
		\|P_{k_1}u^y P_{k_2} v\|_{L^2_{t,x}} &\le  \sum_{m = 1,2}\|P_{k_1}u \Xi_m(D)P_{k_2} v\|_{L^2_{t,x}}\notag\\
		&\le   \sum_{m = 1,2}\Big\|\|P_{k_1}u^y\|_{L^\infty_{\widehat{x}_m}} \|\Xi_m(D)P_{k_2} v\|_{L^2_{\widehat{x}_m}}\Big\|_{L^2_{t,x_m}}\notag\\
		&\lesssim 2^{\frac{k_1}{2}} \sum_{m = 1,2}\Big\|\|P_{k_1}u^y\|_{L^2_{\widehat{x}_m}} \|\Xi_m(D)P_{k_2} v\|_{L^2_{\widehat{x}_m}}\Big\|_{L^2_{t,x_m}},
	\end{align} 
which   can be controlled by Proposition \ref{bilinear1}.
\end{proof}

\section{Estimates in the heat flow direction}
We begin with the following frequency-localized product estimate based on Bony's paraproduct decomposition.
 \begin{lemma}\label{Bony}
	Given Schwartz functions $f,g$, let 
\begin{equation}
  \alpha_k(f) = \sum_{|j-k|\le 20}\|P_j f\|_G, \  \alpha_k(g) = \sum_{|j-k|\le 20}\|P_j g\|_G,
\end{equation}
then we have
\begin{equation}
  \|P_k(f g) \|_{G}\lesssim \sum_{k'\le  k} 2^{k'}[\alpha_{k'}(f)\alpha_k(g) +  \alpha_{k'}(g)\alpha_k(f)] + \sum_{k'\ge k}2^{k} \alpha_{k'}(f)\alpha_{k'}(g).
\end{equation}
 \end{lemma}

\begin{proof}
	We apply the Bony paraproduct decomposition to the product $fg$
\begin{equation}
  P_k(f g)  = \sum_{k_1\le  k-4}^{|k_2-k|\le 4} P_{k}(P_{k_1}f P_{k_2}g )  +\sum_{k_2\le  k-4}^{|k_1-k|\le 4}P_k( P_{k_1}f P_{k_2}g) + \sum_{k_1,k_2\ge  k-4}^{|k_1-k_2|\le 8}P_{k}(P_{k_1}f P_{k_2}g). 
\end{equation}
Following standard terminology, we refer to these terms respectively as the Low-High, High-Low, and High-High paraproduct interactions.

For the Low-High interaction, we use Bernstein-type inequalities to obtain
\begin{equation}\label{lowhigh1}
  \|P_{k_1}f P_{k_2}g\|_{G}\lesssim \|P_{k_1}f\|_{L^\infty_{t,x}}\| P_{k_2}g\|_{G}\lesssim 2^{k_1}\alpha_{k_1}(f)\alpha_{k_2}(g).
\end{equation}
The High-Low case is symmetric.

For the High-High interaction, we have 
\begin{equation}\label{highhigh1}
  \|P_{k_1}f P_{k_2}g\|_{G}\lesssim 2^{k} \|P_{k_1}f P_{k_2}g\|_{L^\infty_t L^1_x \cap  L^4_t L^{\frac{4}{3}}_x}\lesssim 2^{k}\alpha_{k_1}(f)\alpha_{k_2}(g).
\end{equation}

The conclusion follows by summing these estimates over all frequency interactions.
\end{proof}

\begin{lemma}\label{Bony2}
	Let $f(s),g(s)$ be Schwartz functions with the norm $\alpha_k(f(s)),\alpha_k(g(s))$  defined as in Lemma \ref{Bony} . Suppose that  for $s \in [2^{2j-1},2^{2j+1}],\ n \ge 3$ and $ \sigma \in [0,\sigma_1]$, the following bounds hold
\begin{align}
  &\alpha_k(f(s))\lesssim 2^{-\sigma k}\langle 2^{2k}s \rangle^{-n}\beta_{k,j}(\sigma),\ \alpha_k(g(s))\lesssim 2^{-\sigma k}\langle 2^{2k}s \rangle^{-n}\eta_{k,j}(\sigma),  \\
  &\alpha_k(h(s))\lesssim 2^{-(\sigma-1) k}\langle 2^{2k}s \rangle^{-n}\rho_{k,j}(\sigma).
\end{align}
where $\beta_{k,j},\ \eta_{k,j}$ are uniformly slowly varying in $k$. That is, for some sufficiently small $\widetilde{\delta}\ll 1$,
\begin{equation}
  \beta_{k,j}(\sigma) \lesssim 2^{\widetilde{\delta}|k-k'|}\beta_{k',j}(\sigma),\ \forall k,k',j\in \mathbb{Z},  
\end{equation}
and similarly for $\eta_{k,j}(\sigma),\rho_{k,j}(\sigma)$. 

Then, defining $(\beta_{k,j},\eta_{k,j},\rho_{k,j}):= (\beta_{k,j}(0),\eta_{k,j}(0),\rho_{k,j}(0))$, we have the following refined product estimates
\begin{align}
	\|P_k(f(s)g(s))\|_G \lesssim &~  2^{-\sigma k }\min{\{2^{k}, 2^{-j}\}}^{1-2\widetilde{\delta}}2^{-2\widetilde{\delta}j}\langle 2^{2k} s \rangle^{-n}(\beta_{-j,j}\eta_{k,j}(\sigma) +\eta_{-j,j}\beta_{k,j}(\sigma) ), \label{Bony21} \\
	\|P_k(f(s)\pa_x g(s))\|_G & \lesssim  2^{-\sigma k }\langle 2^{2k} s \rangle^{-n} 2^{k-j} (\beta_{-j,j}(\eta_{k,j}(\sigma)+ 1_{\{k+j\le 0\}}\eta_{-j,j}(\sigma)) +\eta_{-j,j}\beta_{k,j}(\sigma)  ), \label{Bony22}\\
	 & \lesssim 2^{-\sigma k }\langle 2^{2k} s \rangle^{-n} 2^{k-j} (\langle 2^{-\frac{k+j}{4}} \rangle\beta_{-j,j} \eta_{k,j}(\sigma) +\eta_{-j,j}\beta_{k,j}(\sigma)  ),\notag \\
	\|P_k(f(s) h(s))\|_G  &\lesssim   2^{-(\sigma-1) k }\langle 2^{2k} s \rangle^{-n} 2^{-j} (  \rho_{-j,j}( \beta_{k,j}(\sigma) + 1_{\{k+j\le 0\}} \beta_{-j,j}(\sigma)  )+\beta_{-j,j}\rho_{k,j}(\sigma)  ).\label{Bony23}
\end{align}
Moreover, for $\sigma\ge \frac{1}{5}$ the bound \eqref{Bony23} can be improved to 
\begin{equation}\label{Bony24}
\|P_k(f(s) h(s))\|_G  \lesssim   2^{-(\sigma-1) k }\langle 2^{2k} s \rangle^{-n} 2^{-j} ( \rho_{-j,j} \beta_{k,j}(\sigma) +\beta_{-j,j}\rho_{k,j}(\sigma)  )
\end{equation}
\end{lemma}
\begin{proof}
	We begin by establishing estimates for the low-frequency sum. We compute:
\begin{align}
	&\sum_{k'\le k}2^{k'}\alpha_{k'}(f(s))\lesssim \sum_{k'\le k} \langle 2^{2k'+2j}  \rangle^{-n}2^{k'}\beta_{k',j}\nonumber\\
	& \le1_{\{k+j\le 0\}}   \sum_{k'\le  k}2^{k'}\beta_{k',j} +   1_{\{k+j\ge 0\}}  \Big( \sum_{k'\le  -j}2^{k'}\beta_{k',j} + \sum_{-j\le k'\le  k}2^{k'}\langle 2^{2(k'+j)} \rangle^{-n}\beta_{k',j}  \Big)\nonumber\\
	&\lesssim 1_{\{k+j\le 0\}} \beta_{-j,j}  \sum_{k'\le  k}2^{k'}2^{-\widetilde{\delta}(k'+j)} \nonumber\\
	&\quad +   1_{\{k+j\ge 0\}} \beta_{-j,j}  \Big( \sum_{k'\le  -j}2^{k'}2^{-\widetilde{\delta}(j+k')} + \sum_{-j\le k'\le  k}2^{k'}\langle 2^{2(k'+j)} \rangle^{-n}2^{\widetilde{\delta}(k'+j)}  \Big)\nonumber\\
	&\lesssim \min{\{2^k, 2^{-j}\}}^{1-\widetilde{\delta}}2^{-\widetilde{\delta}j} \beta_{-j,j}, 
\end{align}
where we use the slow variation property of $\beta_{k,j}$. 
Similarly we have 
\begin{align}
  &\sum_{k'\le k}2^{k'}\alpha_{k'}(f(s))\lesssim \min{\{2^k, 2^{-j}\}}^{1-\widetilde{\delta}}2^{-\widetilde{\delta}j}\eta_{-j,j},\\
  &\sum_{k'\le k}2^{k'}\alpha_{k'}(h(s))\lesssim 2^{k}\min{\{2^k, 2^{-j}\}}^{1-\widetilde{\delta}}2^{-\widetilde{\delta}j}\rho_{-j,j} 
\end{align}

We now apply Lemma \ref{Bony} to estimate the product terms. The estimates split naturally into High-Low, Low-High, and High-High frequency interactions according to the paraproduct decomposition.

\textbf{High-Low and Low-High Interactions:} Consider the sums over \(k_1, k_2\) where one frequency is much lower than \(k\). For the product $P_k( P_{k_1}f(s) P_{k_2}g(s))$, we have 
\begin{align}
	&\sum_{k_1\le  k-4}^{|k_2-k|\le 4}  +\sum_{k_2\le  k-4}^{|k_1-k|\le 4}\|P_k( P_{k_1}f(s) P_{k_2}g(s))\|_G\nonumber\\
    & \lesssim \alpha_k(g(s))\sum_{k'\le  k} 2^{k'} \alpha_{k'}(f(s)) +\alpha_k(f(s))  \sum_{k'\le  k} 2^{k'}\alpha_{k'}(g(s))  \notag\\
	&\lesssim 2^{-\sigma k } \min{\{2^k, 2^{-j}\}}^{1-\widetilde{\delta}}2^{-\widetilde{\delta}j}\langle 2^{2k} s \rangle^{-n}(\beta_{-j,j}\eta_{k,j}(\sigma) +\eta_{-j,j}\beta_{k,j}(\sigma) ).
\end{align}

For the product $ P_k( P_{k_1}f(s) P_{k_2}\pa_x g(s)) $, the additional derivative contributes a factor of $2^k$, which leads to 
\begin{align}
	&\sum_{k_1\le  k-4}^{|k_2-k|\le 4}  +\sum_{k_2\le  k-4}^{|k_1-k|\le 4}\|P_k( P_{k_1}f(s) P_{k_2}\pa_x g(s))\|_G\nonumber\\
    & \lesssim 2^k\alpha_k(g(s))\sum_{k'\le  k} 2^{k'} \alpha_{k'}(f(s)) +2^k \alpha_k(f(s))  \sum_{k'\le  k} 2^{k'}\alpha_{k'}(g(s))  \notag\\
	&\lesssim 2^{-\sigma k }2^{k-j}\langle 2^{2k} s \rangle^{-n}(\beta_{-j,j}\eta_{k,j}(\sigma) +\eta_{-j,j}\beta_{k,j}(\sigma) ).
\end{align}

For the product $ P_k( P_{k_1}f(s) P_{k_2}  h(s)) $, similar computation shows that 
\begin{align}
	&\sum_{k_1\le  k-4}^{|k_2-k|\le 4}  +\sum_{k_2\le  k-4}^{|k_1-k|\le 4}\|P_k( P_{k_1}f(s) P_{k_2} h(s))\|_G\nonumber\\
    & \lesssim  \alpha_k(h(s))\sum_{k'\le  k} 2^{k'} \alpha_{k'}(f(s)) +  \alpha_k(f(s))  \sum_{k'\le  k} 2^{k'}\alpha_{k'}(h(s))  \notag\\
	&\lesssim 2^{-(\sigma-1)k} \min{\{2^k, 2^{-j}\}}^{1-\widetilde{\delta}}2^{-\widetilde{\delta}j}\langle 2^{2k} s \rangle^{-n}(\beta_{-j,j}\rho_{k,j}(\sigma) +\rho_{-j,j}\beta_{k,j}(\sigma) ).\label{Bony241}
\end{align}
which is consistent with the bound \eqref{Bony23} and \eqref{Bony24}.

\textbf{High-High Interactions:} It remains to estimate the sum where both frequencies are high and comparable. For the product $P_k( P_{k_1}f(s) P_{k_2}g(s))$,  we have  
\begin{align}
	&\sum_{k_1,k_2 \ge  k-4}^{|k_1-k_2|\le 8}  \|P_k( P_{k_1}f(s) P_{k_2}g(s))\|_G \lesssim \sum_{k'\ge  k}2^{k}\alpha_{k'}(f(s))\alpha_{k'}(g(s))\nonumber\\
    & \lesssim\sum_{k'\ge  k} 2^{k}2^{-\sigma k'}\langle  2^{2k'+2j} \rangle^{-2n}\beta_{k',j}\eta_{k',j}(\sigma) \notag\\
	&\lesssim 1_{\{k+j \ge 0 \}}\sum_{k'\ge k} \ldots + 1_{\{k+j \le 0 \}}\Big(\sum_{k\le k' \le -j} \ldots+ \sum_{k'\ge -j} \ldots\Big)\notag\\
	&\lesssim 1_{\{k+j \ge 0 \}}2^{-j}2^{-\sigma k}\beta_{-j,j}\eta_{k,j}(\sigma)\sum_{k'\ge k} 2^{k+j} 2^{\widetilde{\delta}(k'-k)} 2^{\widetilde{\delta}(k'+j)}\langle2^{2k'+2j} \rangle^{-2n}\notag\\
	&\quad + 1_{\{k+j \le 0 \}}2^{-\sigma k}\beta_{-j,j}\eta_{k,j}(\sigma)\notag\\
	& \qquad\cdot 2^{k}\Big(\sum_{k\le k' \le -j} 2^{2\widetilde{\delta}(k'-k)}2^{-\widetilde{\delta}(k'+j)} + \sum_{k'\ge -j} 2^{\widetilde{\delta}[(k'-k)+ (k'+j)]} \langle2^{2k'+2j} \rangle^{-2n}\Big)\notag\\
	&\lesssim 1_{\{k+j \ge 0 \}}2^{-\sigma k}2^{-j}\langle2^{2k+2j} \rangle^{2-2n}\beta_{-j,j}\eta_{k,j}(\sigma) +1_{\{k+j \le 0 \}}2^{-\sigma k}2^{(1-2\widetilde{\delta})k-2\widetilde{\delta}j}  \beta_{-j,j}\eta_{k,j}(\sigma),  
\end{align}
which is consistent with the stated bound \eqref{Bony21}.

For $P_k( P_{k_1}f(s) P_{k_2}\pa_x g(s))$ in the High-High regime, the derivative gives an extra factor of $2^{k'}$, which leads to
\begin{align}
	&\sum_{k_1,k_2 \ge  k-4}^{|k_1-k_2|\le 8}  \|P_k( P_{k_1}f(s) P_{k_2}\pa_x g(s))\|_G \lesssim \sum_{k'\ge  k}2^{k+k'}\alpha_{k'}(f(s))\alpha_{k'}(g(s))\notag\\
    & \lesssim\sum_{k'\ge  k} 2^{k+k'}2^{-\sigma k'}\langle  2^{2k'+2j} \rangle^{-2n}\beta_{k',j}\eta_{k',j}(\sigma) \notag\\
	&\lesssim 1_{\{k+j \ge 0 \}}\sum_{k'\ge k} \ldots + 1_{\{k+j \le 0 \}}\Big(\sum_{k\le k' \le -j} \ldots+ \sum_{k'\ge -j} \ldots\Big)\notag\\
	&\lesssim 1_{\{k+j \ge 0 \}}2^{k-j}2^{-\sigma k}\beta_{-j,j}\eta_{k,j}(\sigma)\sum_{k'\ge k} 2^{k'+j} 2^{\widetilde{\delta}(k'-k)} 2^{\widetilde{\delta}(k'+j)}\langle2^{2k'+2j} \rangle^{-2n}\notag\\
	&\quad + 1_{\{k+j \le 0 \}}2^{-\sigma k}\beta_{-j,j}\eta_{-j,j}(\sigma)\notag\\
	& \qquad\cdot 2^{k-j}\Big(\sum_{k\le k' \le -j} 2^{k'+j}2^{-2\widetilde{\delta}(k'+j)} + \sum_{k'\ge -j}2^{k'+j} 2^{2\widetilde{\delta}(k'+j)} \langle2^{2k'+2j} \rangle^{-2n}\Big)\notag\\
	&\lesssim 1_{\{k+j \ge 0 \}}2^{k-j}2^{-\sigma k}\langle 2^{2k+2j} \rangle^{2-2n}\beta_{-j,j}\eta_{k,j}(\sigma) + 1_{\{k+j \le 0 \}}2^{k-j}2^{-\sigma k} \beta_{-j,j}\eta_{-j,j}(\sigma), \notag\\
	&\lesssim 1_{\{k+j \ge 0 \}}\ldots + 1_{\{k+j \le 0 \}}2^{k-j}2^{-\sigma k} 2^{-\widetilde{\delta}(k+j)}\beta_{-j,j}\eta_{k,j}(\sigma),\label{320}
\end{align}
matching  the stated bound \eqref{Bony22}. 

For the $P_k( P_{k_1}f(s) P_{k_2} h(s))$ in the High-High regime, similar computation as in \eqref{320}  shows that 
\begin{align}
	&\sum_{k_1,k_2 \ge  k-4}^{|k_1-k_2|\le 8}  \|P_k( P_{k_1}f(s) P_{k_2} h(s))\|_G \lesssim \sum_{k'\ge  k}2^{k }\alpha_{k'}(f(s))\alpha_{k'}(h(s))\notag\\
    & \lesssim\sum_{k'\ge  k} 2^{k+k'}2^{-\sigma k'}\langle  2^{2k'+2j} \rangle^{-2n}\beta_{k',j}(\sigma)\rho_{k',j} \notag\\
	&\lesssim 1_{\{k+j \ge 0 \}}2^{-j}2^{-(\sigma-1) k}\langle 2^{2k+2j} \rangle^{2-2n}\rho_{-j,j}\beta_{k,j}(\sigma)\notag\\
	&\quad\quad + 1_{\{k+j \le 0 \}}2^{-j}2^{-(\sigma-1) k}  \rho_{-j,j}\beta_{-j,j}(\sigma).\label{Bony242}
\end{align}

To obtain \eqref{Bony24} for $\sigma\ge  \frac{1}{5}$, in view of \eqref{Bony241} and \eqref{Bony242} it is sufficient to improve the bound in the High-High interaction when $k+j\le 0$. We have 
\begin{align}
	&1_{\{k+j\le 0\}}\sum_{k_1,k_2 \ge  k-4}^{|k_1-k_2|\le 8}  \|P_k( P_{k_1}f(s) P_{k_2} h(s))\|_G \lesssim 1_{\{k+j\le 0\}}\sum_{k'\ge  k}2^{k }\alpha_{k'}(f(s))\alpha_{k'}(h(s))\notag\\
    & \lesssim 1_{\{k+j\le 0\}} \sum_{k'\ge  k} 2^{k+k'}2^{-\sigma k'}\langle  2^{2k'+2j} \rangle^{-2n}\beta_{k',j}(\sigma)\rho_{k',j} \notag\\
	&\lesssim 1_{\{k+j\le 0\}} \beta_{-j,j}(\sigma)\rho_{-j,j} 2^{ k}\Big(\sum_{k'\ge  -j} 2^{(1-\sigma)k'} \langle  2^{2k'+2j} \rangle^{-2n}2^{2\widetilde{\delta}(k'+j)}  + \sum_{k \le  k'\le  -j} 2^{(1-\sigma)k'}  2^{-2\widetilde{\delta}(k'+j)} \Big)\notag\\
	&\lesssim 1_{\{k+j\le 0\}} \beta_{-j,j}(\sigma)\rho_{-j,j} 2^{ k} (\max\{2^{-(1-\sigma)j},2^{(1-\sigma)k}2^{-2\widetilde{\delta}(k+j)}\})  \notag\\
	&\lesssim1_{\{k+j\le 0\}} \beta_{-j,j}(\sigma)\rho_{-j,j} 2^{ k}2^{-(1-\sigma)j}  \langle 2^{(1-\sigma-2\widetilde{\delta})(k+j)}  \rangle  \notag\\
	&\lesssim \notag  1_{\{k+j\le 0\}} 2^{ \sigma (k+j)}\langle 2^{(1-\sigma-2\widetilde{\delta})(k+j)}  \rangle 2^{-j} 2^{-(\sigma-1)k}\rho_{-j,j}\beta_{-j,j}(\sigma)\notag\\
	&\lesssim 1_{\{k+j\le 0\}} 2^{\frac{1}{5}(k+j)}2^{-j} 2^{-(\sigma-1)k}\rho_{-j,j}\beta_{-j,j}(\sigma)\lesssim 1_{\{k+j\le 0\}}   2^{-j} 2^{-(\sigma-1)k}\beta_{k,j}(\sigma)\rho_{-j,j},\notag
\end{align}
which suffices to prove the bound \eqref{Bony24}.

This completes the proof.
\end{proof}

We now apply these technical lemmas to obtain estimates for the differentiated fields and connection coefficients $\psi_\alpha, A_\alpha$ in the heat flow direction.

By \eqref{afre} we have the following bound for $\psi_\alpha$ when   $s \in [2^{2j-1},2^{2j+1}]$: 
\begin{align}
	  &\|P_k \bm{\psi}(s)\|_{G}\le   2^{-\sigma k}\langle 2^{2k}s \rangle^{-4} a_{k}(\sigma),\ \text{for}\ \sigma \in [0,\sigma_1],\label{psi-bound}\\
	   &\|P_k \psi_t(s)\|_{G}\le   2^{-(\sigma-1)k}\langle 2^{2k}s \rangle^{-4} a_{k}(\sigma),\ \text{for}\ \sigma \in [\frac{1}{5},\sigma_1].\label{psit-bound}
	\end{align}
	
Define $B_1$ as the smallest constant in $[1,\infty)$ such that
	\begin{equation}\label{A-bound}
	  \|P_k \bm{A}(s)\|_{G}\lesssim B_1 2^{-\sigma k}\langle 2^{2k}s \rangle^{-\frac{7}{2}} a_{k,j}(\sigma),
	\end{equation}
where $a_{k,j}(\sigma)$ is defined by
\begin{equation}
  a_{k,j}(\sigma) :=a_{\min(-j,k)}a_{k}(\sigma). 
\end{equation}

We note that $a_{k,j}(\sigma)$ is uniformly slowly varying in $k$. Using the slow variation of $a_k$, we have the lower bound
\begin{equation}\label{low}
  \langle 2^{-\delta(k+j)} \rangle^{-1} a_{-j}a_k(\sigma)\lesssim   1_{\{k+j\le  0\}}a_k a_k(\sigma)+ 1_{\{k+j\ge  0\}}  a_{-j}a_{k}(\sigma)=  a_{k,j}(\sigma) 
\end{equation}  
and  the upper bound 
\begin{equation} \label{up}
   a_{k,j}(\sigma)  \lesssim  \langle 2^{-\delta(k+j)} \rangle a_{-j} a_{k}(\sigma).
\end{equation}
The uniform slow variation of $a_{k,j}(\sigma)$ follows from that of $a_k(\sigma)$.

With these preliminaries, we apply Lemma \ref{Bony2} to control nonlinear interactions involving $\psi_\alpha$ and $A_\alpha$.
\begin{lemma}
	 Let  $s \in [2^{2j-1},2^{2j+1}], \sigma \in [0,\sigma_1]$ 
and assume the bound \eqref{A-bound} holds. 
Then we have the following bilinear and trilinear estimates 
\begin{align}
	&  \|P_k(\bm{\psi}(s)\cdot \bm{\psi}(s))\|_{G}\lesssim 2^{-\sigma k} \min{\{2^{k}, 2^{-j}\}}^{1-2\delta}2^{-2\delta j}   \langle 2^{2k}s \rangle^{-4}a_{-j}a_k(\sigma),\label{b1}\\
	& \|P_k(\bm{\psi}(s)\cdot \pa_x \bm{\psi}(s))\|_{G}\lesssim 2^{-\sigma k}\langle 2^{2k}s \rangle^{-4} 2^{k-j}\langle (2^{2k}s)^{-\frac{1}{8}} \rangle a_{-j}a_k(\sigma) ,\label{b2}\\
	&\|P_k(\bm{A}(s)\cdot \bm{\psi}(s))\|_G \lesssim B_1 2^{-\sigma k}  \min{\{2^{k}, 2^{-j}\}}^{1-6\delta}2^{-6\delta j} \langle 2^{2k} s \rangle^{-\frac{7}{2}}a_{-j}^2  a_{k}(\sigma), \label{b3}\\
	&\|P_k(\bm{A}(s)\cdot \bm{A}(s))\|_G \lesssim B^2_1 2^{-\sigma k}  \min{\{2^{k}, 2^{-j}\}}^{1-6\delta}2^{-6\delta j} \langle 2^{2k} s \rangle^{-\frac{7}{2}}a_{-j}^3 a_{k}(\sigma),\label{b4}\\
	&\|P_k(\bm{A}(s)\cdot \pa_x \bm{\psi}(s))\|_G \lesssim  B_1 2^{-\sigma k}2^{k-j}\langle 2^{2k} s \rangle^{-\frac{7}{2}}\langle (2^{2k}s)^{-\frac{1}{8}} \rangle a_{-j}^2 a_{k}(\sigma),  \label{b5}\\	
	& \|P_k(\bm{A}(s)\cdot \psi_t(s))\|_G \lesssim  B_1 2^{-\sigma k}2^{k-j}\langle 2^{2k} s \rangle^{-\frac{7}{2}}\langle (2^{2k}s)^{-\frac{1}{8}} \rangle a_{-j}^2 a_{k}(\sigma),\ \sigma\ge \frac{1}{5} ,\label{b8}\\	
	&\|P_{k}(\bm{A}(s)\cdot\bm{A}(s)\cdot \bm{\psi}(s))\|_{G}+\|P_{k}(\bm{\psi}(s)\cdot\bm{\psi}(s)\cdot \bm{\psi}(s))\|_{G}  \label{b7}\\
	&\qquad\quad  \lesssim B^2_1 2^{-\sigma k} 2^{-2j}\langle 2^{2k} s \rangle^{-\frac{7}{2}}(a_{-j}^4 + a_{-j}^2) a_{k}(\sigma), \notag\\
	&\|P_{k}(\bm{A}(s)\cdot\bm{\psi}(s)\cdot \bm{\psi}(s))\|_{G} \label{b6}\\
	&\qquad\quad \lesssim  B_1 2^{-\sigma k}(2^{2k}2^{-\frac{3(k+j)}{2}}1_{\{k+j\le 0\}} + 2^{-2j}1_{\{k+j\ge 0\}} )\langle 2^{2k} s \rangle^{-\frac{7}{2}} a_{-j}^3a_{k}(\sigma) .  \notag
\end{align}
\end{lemma}
\begin{proof}
Bounds \eqref{b1} and \eqref{b2} follow directly by applying Lemma \ref{Bony2} to the products involving the bound \eqref{psi-bound}.

To obtain bounds \eqref{b3}--\eqref{b5}, we apply Lemma \ref{Bony2} to the product forms of the bounds \eqref{psi-bound}, \eqref{psit-bound} and \eqref{A-bound}, obtaining
\begin{align}
	&\|P_k(\bm{A}(s)\cdot \bm{\psi}(s))\|_G  \label{b31}\\
	&\quad \lesssim B_1 2^{-\sigma k}\min{\{2^{k}, 2^{-j}\}}^{1-4\delta}2^{-4\delta j}\langle 2^{2k} s \rangle^{-\frac{7}{2}}a_{-j}(  a_{k,j}(\sigma) + a_{-j} a_k(\sigma)) \notag\\
	&\|P_k(\bm{A}(s)\cdot \bm{A}(s))\|_G \lesssim B^2_1 2^{-\sigma k} \min{\{2^{k}, 2^{-j}\}}^{1-4\delta}2^{-4\delta j}\langle 2^{2k} s \rangle^{-\frac{7}{2}}a_{-j}^2 a_{k,j}(\sigma),\label{b41}\\
	&\|P_k(\bm{A}(s)\cdot \pa_x \bm{\psi}(s))\|_G+ \|P_k(\bm{A}(s)\cdot \psi_t(s))\|_G\label{b51}\\
	& \quad\lesssim  B_1 2^{-\sigma k}2^{k-j}\langle 2^{2k} s \rangle^{-\frac{7}{2}}a_{-j}(\langle 2^{-\frac{k+j}{4}} \rangle a_{-j}  a_{k}(\sigma) +  a_{k,j}(\sigma) )\notag.
\end{align}
Then \eqref{b3} and \eqref{b4} follow from \eqref{b31}, \eqref{b41} together with  the inequality 
\begin{equation}
  \begin{aligned}
   &\min\{2^{k},2^{-j}\}^{2\delta} a_{k,j}(\sigma) \lesssim \min\{2^{k},2^{-j}\}^{2\delta}  \langle 2^{-2\delta(k+j)} \rangle a_{-j} a_k(\sigma)\lesssim 2^{-2\delta j}a_{-j}a_k(\sigma), 
  \end{aligned}
\end{equation}
while \eqref{b5} follows from \eqref{b51} and the upper bound \eqref{up} for $a_{k,j}(\sigma)$.

The bound \eqref{b7} is obtained by applying Lemma \ref{Bony2} to \eqref{psi-bound}, \eqref{b1}, and \eqref{b4} while  replacing the factor $\min\{2^{k},2^{-j}\}$ with $2^{-j}$ .

 For \eqref{b6}, we apply Lemma \ref{Bony2} to the product $\bm{A}(s)\cdot(\bm{\psi}(s)\cdot \bm{\psi}(s))$. Using \eqref{b1} and the lower bound \eqref{low} for $a_{k,j}(\sigma)$, we have
\begin{align} 
   \|P_k(\bm{\psi}(s)\cdot \bm{\psi}(s))\|_{G}&\lesssim 2^{-\sigma k} \min{\{2^{k}, 2^{-j}\}}^{1-2\delta}2^{-2\delta j}\langle 2^{2k}s \rangle^{-4}a_{-j}a_k(\sigma)\notag\\
   &\lesssim 2^{-\sigma k} \min{\{2^{k}, 2^{-j}\}}^{1-2\delta}2^{-2\delta j}\langle 2^{2k}s \rangle^{-4}\langle 2^{-2\delta(k+j)} \rangle a_{k,j}(\sigma)\notag\\
   &\lesssim 2^{-\sigma k} 2^{-j}\langle 2^{2k}s \rangle^{-4} a_{k,j}(\sigma).
\end{align}
Combining this with \eqref{A-bound} and \eqref{Bony21} yields
	\begin{align}
		&\|P_{k}(\bm{A}(s)\cdot\bm{\psi}(s)\cdot \bm{\psi}(s))\|_{G}\notag\\
		&\lesssim B_1 2^{-\sigma k}\min{\{2^{k}, 2^{-j}\}}^{1-4\delta}2^{-4\delta j}\langle 2^{2k} s \rangle^{-\frac{7}{2}} \cdot 2^{-j}a_{-j,j} a_{k,j}(\sigma) \notag\\
		&\lesssim B_1 2^{-\sigma k}2^{-j}\min{\{2^{k}, 2^{-j}\}}^{1-4\delta}2^{-4\delta j}\langle 2^{2k} s \rangle^{-\frac{7}{2}} a_{-j}^2   a_{k,j}(\sigma) \notag\\
		&\lesssim B_1 2^{-\sigma k}2^{-j}\min{\{2^{k}, 2^{-j}\}}^{1-6\delta}2^{-6\delta j}\langle 2^{2k} s \rangle^{-\frac{7}{2}} a_{-j}^3 a_k(\sigma)\notag\\
		&\lesssim B_1 2^{-\sigma k}(2^{2k}2^{-\frac{3(k+j)}{2}}1_{\{k+j\le 0\}} + 2^{-2j}1_{\{k+j\ge 0\}} )\langle 2^{2k} s \rangle^{-\frac{7}{2}} a_{-j}^3a_{k}(\sigma) .
	\end{align}
 This completes the proof.
\end{proof}

We now establish the main estimate for the connection coefficient $\bm{A}(s)$.

\begin{proposition}
  Let  $s \in [2^{2j_0-1},2^{2j_0+1}], \sigma \in [0,\sigma_1]$. Under the bootstrap assumptions in Proposition \ref{heatBA} we have 
  \begin{equation}\label{Ab}
	 \|P_k \bm{A}(s)\|_{G}\lesssim 2^{-\sigma k}\langle 2^{2k}s \rangle^{-\frac{7}{2}} a_{k,j_0}(\sigma).
  \end{equation}
\end{proposition}

\begin{proof}
We start from the identity \eqref{Aid}
\begin{equation} \label{Amid}
  A_m(s) = -\int_{s}^{+\infty}\Im(\overline{\psi_m}D_l \psi_l )(r)\mathrm{d}r,\ m = 1,2.
\end{equation}

Applying $P_k$ to \eqref{Amid} yields 
\begin{equation}\label{estimateA}\begin{aligned}
	\|P_k \bm{A}(s)\|_G \lesssim &~ \int_{s}^{+\infty} \|P_{k}(\bm{\psi}(r) \cdot \pa_x \bm{\psi}(r))\|_G    +  \|P_{k}(\bm{A}(r)\cdot \bm{\psi}(r) \cdot  \bm{\psi}(r))\|_G \mathrm{d} r.  
\end{aligned}
\end{equation}

For the first term, using \eqref{b2} we compute
\begin{align}
	&\int_{s}^{+\infty} \|P_{k}(\bm{\psi}(r) \cdot \pa_x \bm{\psi}(r))\|_G \mathrm{d} r \notag\\
	& \lesssim \sum_{j\ge  j_0}2^{2j} \cdot 2^{-\sigma k}\langle 2^{2k+2j} \rangle^{-4} 2^{k-j} \langle 2^{-\frac{k+j}{4}} \rangle a_{-j}a_k(\sigma) \notag\\
	&\lesssim 1_{\{k+j_0\ge 0\}}\sum_{j\ge  j_0}\ldots + 1_{\{k+j_0\le 0\}}\Big(\sum_{j_0\le j\le  -k}\ldots + \sum_{  j\ge  -k}\ldots \Big)\notag\\
	&\lesssim 1_{\{k+j_0\ge 0\}}2^{-\sigma k}a_{-j_0}a_k(\sigma)\sum_{j\ge  j_0}\langle 2^{2k+2j} \rangle^{-4}2^{k+j} 2^{\delta(j-j_0)}\notag\\
	&\  + 1_{\{k+j_0\le 0\}}2^{-\sigma k}a_{k}a_k(\sigma)\Big(\sum_{j_0\le j\le  -k}2^{k+j}2^{-(\frac{1}{4}+\delta)(k+j)} +\sum_{  j\ge  -k}\langle 2^{2k+2j} \rangle^{-4} 2^{k+j}2^{\delta(k+j)}  \Big) \notag\\
	&\lesssim 1_{\{k+j_0\ge 0\}}2^{-\sigma k}\langle 2^{2k+2j_0} \rangle^{-\frac{7}{2}} a_{-j_0}a_k(\sigma) + 1_{\{k+j_0\le 0\}} 2^{-\sigma k}a_k a_k(\sigma)\notag\\
	&\lesssim 2^{-\sigma k} \langle 2^{2k+2j_0} \rangle^{-\frac{7}{2}}  a_{k,j_0}(\sigma).\label{Ab1}
\end{align}

For the second term, using \eqref{b6} and the bootstrap assumption $\sum_k a_k^2 \le \varepsilon$,
\begin{align}
	&\int_{s}^{+\infty} \|P_{k}(\bm{A}(r)\cdot \bm{\psi}(r) \cdot  \bm{\psi}(r))\|_G \mathrm{d} r\notag\\
	& \lesssim \sum_{j\ge  j_0} B_1 2^{-\sigma k} (2^{ \frac{ k+j }{2}}1_{\{k+j\le 0\}} + 1_{\{k+j\ge 0\}} )\langle 2^{2k+2j}  \rangle^{-\frac{7}{2}} a_{-j}^3  a_k(\sigma)\notag\\
	&\lesssim 1_{\{k+j_0 \le  0 \}}B_1\varepsilon 2^{-\sigma k}a_k a_k(\sigma)\big(\sum_{j\ge -k}\langle 2^{2k+2j}  \rangle^{-\frac{7}{2}} 2^{\delta(k+j)} + \sum_{j_0\le j\le  -k}2^{ k+j }2^{-(\frac{1}{4} + 2\delta)(k+j)}\big)\notag\\
	&\quad + 1_{\{k+j_0\ge 0\}}B_1\varepsilon 2^{-\sigma k}\sum_{j\ge j_0}\langle 2^{2k+2j}  \rangle^{-\frac{7}{2}} 2^{\delta(j-j_0)}a_{-j_0}a_k(\sigma)\notag\\
	&\lesssim B_1\varepsilon2^{-\sigma k}( 1_{\{k+j_0 \le  0 \}} a_{k}a_{k}(\sigma)+ 1_{\{k+j_0 \ge  0 \}}\langle 2^{2k+2j_0}  \rangle^{-\frac{7}{2}} a_{-j_0}a_k(\sigma))\notag\\
	&\lesssim B_1\varepsilon 2^{-\sigma k} \langle 2^{2k+2j_0}  \rangle^{-\frac{7}{2}}a_{k,j_0}(\sigma), \label{Ab2}
\end{align}

Combining \eqref{Ab1} and \eqref{Ab2}, we have 
\begin{equation}
  \|P_k \bm{A}(s)\|_{G}\lesssim \langle B_1 \varepsilon \rangle  2^{-\sigma k} \langle 2^{2k } s \rangle^{-\frac{7}{2}}a_{k,j_0}(\sigma),
\end{equation}
which implies $B_1\lesssim 1 + B_1\varepsilon$, hence $B_1 \lesssim 1$ since $\varepsilon$ is sufficiently small.
\end{proof}

We next establish the bound for $\bm{\psi}(s)$ and $\psi_t(s)$ using the heat equation.

\begin{proposition}
	Under the bootstrap assumptions in Proposition \ref{heatBA}, we have 
	\begin{align}
		&\|P_k \bm{\psi}(s)\|_G \lesssim 2^{-\sigma k}\langle  2^{2k }s\rangle^{-4} b_k(\sigma),\ \text{for}\ \sigma \in [0,\sigma_1],\\
		& \|P_k \psi_t(s)\|_G \lesssim 2^{k}\langle  2^{2k }s\rangle^{-4}\varepsilon^{\frac{1}{2}},\ 2^{-(\sigma-1) k}\langle  2^{2k }s\rangle^{-4} b_k(\sigma),\ \text{for}\ \sigma \in [\frac{1}{5},\sigma_1].\label{psitk}
	\end{align}
\end{proposition}
\begin{proof}
	We use the heat equation for $\psi_\alpha$:
	\begin{equation}
		\begin{aligned}
		&  (\pa_s-\Delta_x)\psi_\alpha = K(\psi_\alpha),\\
		\text{where}&~ K(\psi):= 2i A_l \pa_l\psi- (A_l^2 -i\pa_lA_l) \psi + i \Im(\psi \overline{\psi_l})\psi_l ,
		\end{aligned}
	\end{equation}
which gives the Duhamel formula
\begin{equation}\label{heatformula}
  \psi_\alpha(s) = e^{s\Delta_x}\psi_\alpha(0) + \int_0^s e^{(s-r)\Delta_x}K(\psi_\alpha)\mathrm{d}r.
\end{equation}

For $\bm{\psi}$, we rewrite $K(\bm{\psi})$ as
\begin{equation}
  K(\bm{\psi})   =  \pa_x(\bm{A} \bm{\psi}) + \bm{A}\pa_x\bm{\psi} + (\bm{A}\cdot \bm{A} + \bm{\psi}\cdot \bm{\psi})\psi.
\end{equation} 

Using \eqref{b3}, \eqref{b5}, \eqref{b7} and the bootstrap assumption $\sum_k a_k^2 \le \varepsilon$, for $r \in [2^{2j-1},2^{2j+1}]$ we obtain
\begin{align}\label{Kpsi-bound}
	  &\|P_k K(\bm{\psi})(s)\|_G \lesssim   2^{-\sigma k} (2^{-2j}a_{-j}^2 +  \varepsilon 2^{k-j}\langle 2^{-\frac{k+j}{2}}  \rangle) \langle 2^{2k+2j}  \rangle^{-\frac{7}{2}}  a_{k}(\sigma).
\end{align}

 Assume $s\in [2^{2k_0-1},2^{2k_0+1}]$. For $k+k_0\ge 0$, splitting the time integral at $s/2$, we have
 \begin{align}
	 &\Big\|\int_0^s e^{(s-r)\Delta_x}P_k K(\psi_\alpha)\mathrm{d}r\Big\|_{G}\notag\\
	 &\lesssim \int_{0}^{\frac{s}{2}}\langle 2^{2k}s \rangle^{-N}\|P_k K(\psi_\alpha)(r)\|_G\mathrm{d}r + \int_{\frac{s}{2}}^s \langle 2^{2k}(s-r) \rangle^{-N}\|P_k K(\psi_\alpha)(r)\|_G\mathrm{d}r\notag\\
	 &\lesssim   2^{-\sigma k}\langle 2^{2k}s \rangle^{-N} a_k(\sigma)\sum_{j\le k_0} (a_{-j}^2+ \varepsilon 2^{k+j}\langle 2^{-\frac{k+j}{2}} \rangle)\langle 2^{2k+2j}  \rangle^{-\frac{7}{2}}\notag\\
	 &\quad +   \varepsilon 2^{-\sigma k} (2^{-2k-2k_0} + 2^{-k-k_0} ) \langle 2^{2k+2k_0}  \rangle^{-\frac{7}{2}}  a_{k}(\sigma)\notag \\
	 &\lesssim \varepsilon 2^{-\sigma k}\langle 2^{2k+2k_0}  \rangle^{-4}a_k(\sigma).
\end{align}
 For $k+k_0\le 0$, we integrate directly
\begin{align}
	 &\Big\|\int_0^s e^{(s-r)\Delta_x}P_k K(\psi_\alpha)\mathrm{d}r\Big\|_{G} \lesssim \sum_{j\le k_0}\int_{2^{2j-1}}^{2^{2j+1}} \|P_k K(\psi_\alpha)(r)\|_G\mathrm{d}r  \notag\\
	 &\lesssim 2^{-\sigma k}a_{k}(\sigma) \sum_{j\le  k_0}( a_{-j}^2 +  \varepsilon 2^{\frac{k+j}{2}} )  \lesssim \varepsilon 2^{-\sigma k}a_{k}(\sigma).
\end{align}

 Combining both cases  yields
\begin{equation}
  \Big\|\int_0^s e^{(s-r)\Delta_x}P_k K(\psi_\alpha)\mathrm{d}r\Big\|_{G}\lesssim \varepsilon 2^{-\sigma k}\langle 2^{2k }s  \rangle^{-4}a_k(\sigma).
\end{equation}

From \eqref{heatformula} and the definition of $b_k(\sigma)$,
\begin{align}
	\|P_k \bm{\psi}(s)\|_{G}&\lesssim \|P_k e^{s\Delta_x}\bm{\psi}(0)\|_G + \Big\|\int_0^s e^{(s-r)\Delta_x}P_k K(\bm{\psi})\mathrm{d}r\Big\|_{G}\notag\\
  &\lesssim 2^{-\sigma k}( \langle 2^{2k}s \rangle^{-N}b_{k}(\sigma) + \varepsilon \langle 2^{2k}s \rangle^{-4}a_{k}(\sigma) )\notag\\
  &\lesssim  2^{-\sigma k} \langle 2^{2k}s \rangle^{-4} (b_k(\sigma) + \varepsilon a_k(\sigma)).\label{psi-b}
 \end{align}

For $\psi_t(s)$, let $B_2$ be the smallest number in $[1,\infty)$ such that
\begin{equation}\label{psit-b2}
  \|P_k \psi_t(s)\|_G \lesssim B_2 \varepsilon^{\frac{1}{2}} 2^k\langle 2^{2k} s \rangle^{-4}, 
\end{equation}

we rewrite $K(\psi_t)$ as 
\begin{equation}
   K(\psi_t)   =  \pa_x(\bm{A} \psi_t) + [\pa_x\bm{A} + (\bm{A}\cdot \bm{A} + \bm{\psi}\cdot \bm{\psi})]\psi_t.
\end{equation}

Using  \eqref{A-bound}, \eqref{up}, \eqref{b1}, \eqref{b4} and the slow variation of $a_k(\sigma)$, we have 

  \begin{align}
  \|P_k(\bm{A}(s)\cdot\bm{A}(s) ,\ \bm{\psi}(s)\cdot\bm{\psi}(s),\ \pa_x \bm{A}(s) ) \|_G &\lesssim 2^{-j}   \langle 2^{2k} s \rangle^{-\frac{7}{2}}a_{-j}^2  .\label{psit-b1}\\
 \text{or} &\lesssim 2^{-\sigma k}2^{-j}  \langle 2^{2k} s \rangle^{-\frac{7}{2}}    a_{-j} a_{k}(\sigma)\label{psit-b11}
  \end{align}

Applying \eqref{Bony23} to \eqref{psit-b2} and \eqref{psit-b1} gives 
 \begin{equation}
  \|P_k [(\bm{A}(s)\cdot\bm{A}(s) ,\ \bm{\psi}(s)\cdot\bm{\psi}(s),\ \pa_x \bm{A}(s) )\cdot \psi_t(s) ]\|_G \lesssim B_2 \varepsilon^{\frac{1}{2}}  2^k 2^{-2j} \langle 2^{2k} s \rangle^{-\frac{7}{2}} a_{-j}^2.
 \end{equation}

For $\sigma \ge \frac{1}{5}$, applying  \eqref{Bony24} to \eqref{psit-bound} and \eqref{psit-b11} gives
\begin{equation}
  \begin{aligned}
  &\|P_k [(\bm{A}(s)\cdot\bm{A}(s) ,\ \bm{\psi}(s)\cdot\bm{\psi}(s),\ \pa_x \bm{A}(s) )\cdot \psi_t(s)] \|_G \\
  &\lesssim  2^{-(\sigma-1)k} 2^{-2j} \langle 2^{2k} s \rangle^{-\frac{7}{2}} a_{-j}^2 a_{k}(\sigma) .
  \end{aligned}
\end{equation}

Combining this with \eqref{b8} yields
\begin{equation}
  	\begin{aligned}
		\|P_k K(\psi_t)\|_G&\lesssim 2^{-(\sigma-1)k} (2^{-2j}a_{-j}^2 +  \varepsilon 2^{k-j}\langle 2^{-\frac{k+j}{2}}  \rangle)\langle 2^{2k} s \rangle^{-\frac{7}{2}} a_{k}(\sigma),\ \text{for}\ \sigma\ge  \frac{1}{5}\\
		 \text{or}\ &\lesssim B_2 \varepsilon^{\frac{1}{2}}  2^k (2^{-2j}a_{-j}^2 +   2^{k-j}\langle 2^{-\frac{k+j}{2}}  \rangle) \langle 2^{2k} s \rangle^{-\frac{7}{2}} ,
	\end{aligned}
\end{equation}
which takes a form similar to \eqref{Kpsi-bound}.  

Repeating the argument used for $\bm{\psi}$ gives
\begin{equation}
	\begin{aligned}
	  \|P_k\psi_t(s)\|_G &\lesssim  2^{-(\sigma-1) k} \langle 2^{2k}s \rangle^{-4} (b_k(\sigma) + \varepsilon a_k(\sigma)),\ \text{for}\ \sigma \ge  \frac{1}{5}, \\
	  \text{or} &\lesssim  (1 + B_2\varepsilon)2^{ k} \langle 2^{2k}s \rangle^{-4}  \varepsilon^{\frac{1}{2}} .
	\end{aligned}\label{psit-b}
\end{equation}

By definition \eqref{afre}, \eqref{psit-b} and the above imply $a_k(\sigma)\lesssim b_k(\sigma) + \varepsilon a_k(\sigma),\ B_2 \lesssim (1 + B_2\varepsilon)$, hence  $a_k(\sigma)\lesssim b_k(\sigma),\ B_2 \lesssim 1$ since $\varepsilon$ is sufficiently small.  This completes the proof.
\end{proof}

Finally, we establish bilinear estimates involving $\bm{\psi}(s)$.

\begin{proposition} Under the bootstrap assumptions in Proposition \ref{heatBA} we have
\begin{equation}\label{psis-bilinear1}
	\begin{aligned}
	  &\sup_y \Big\|\|P_{k} \bm{\psi}(s)^y\|_{L^2_{\widehat{x}_m}} \|\Xi_m(D)P_{k'} w\|_{L^2_{\widehat{x}_m}}\Big\|_{L^2_{t,x_m}}\\
	  &\ \lesssim 2^{-\sigma k}2^{-\frac{k'}{2}}\langle 2^{2k}s \rangle^{-3} b_k(\sigma) \mathcal{D}(P_{k'}w).\ \text{for}\ k'\ge k + 100,
	\end{aligned}
\end{equation}
and 
\begin{equation}\label{psis-bilinear2}
	\begin{aligned}
	  &\sup_y \|P_{k}\bm{\psi}(s)P_{k'} w^y\|_{L^2_{t,x}}\\
	  &\ \lesssim 2^{-\sigma k}2^{\frac{k'-k}{2}}\langle 2^{2k}s \rangle^{-3} b_k(\sigma) \mathcal{D}(P_{k'}w).\ \text{for}\ k\ge k'-100.
	\end{aligned}
\end{equation}
Moreover, we have 
	\begin{align}\label{psis-main}
	 &\sup_y \|P_{k}\bm{\psi}(s) P_{k'}w^y\|_{L^2_{t,x}}\lesssim 2^{-\frac{|k -k'|}{2}} 2^{-\sigma k} \langle 2^{2k}s \rangle^{-3} b_k(\sigma)\mathcal{D}(P_{k'}w),\\
	 &\label{psis-main-2}\sup_y \|P_{k}D_x \bm{\psi}(s) P_{k'}w^y\|_{L^2_{t,x}} \lesssim 2^{-\frac{|k -k'|}{2}} 2^{-(\sigma-1)k} (2^{2k}s)^{-\frac{1}{4}}\langle 2^{2k}s \rangle^{-3}b_{k}(\sigma)\mathcal{D}(P_{k'}w).  
	\end{align}
\end{proposition}
\begin{proof}
For \eqref{psis-bilinear1}, let $B_3$ be the smallest number in $[1,\infty)$ such that 
\begin{equation} 
	\begin{aligned}
	  &\sup_y \Big\|\|P_{k} \bm{\psi}(s)^y\|_{L^2_{\widehat{x}_m}} \|\Xi_m(D)P_{k'} w\|_{L^2_{\widehat{x}_m}}\Big\|_{L^2_{t,x_m}}\\
	  &\ \lesssim  B_3 2^{-\sigma k}2^{-\frac{k'}{2}}\langle 2^{2k}s \rangle^{-3} b_k(\sigma) \mathcal{D}(P_{k'}w).\ \text{for}\ k'\ge k + 100,
	\end{aligned}
\end{equation}

we use \eqref{heatformula} to expand
\begin{align}
	&\sup_y \Big\|\|P_{k} \bm{\psi}(s)^y\|_{L^2_{\widehat{x}_m}} \|\Xi_m(D)P_{k'} w\|_{L^2_{\widehat{x}_m}}\Big\|_{L^2_{t,x_m}}\notag\\
	& = \sup_y \Big\|\|P_{k} e^{s\Delta_x}\bm{\psi}(0)^y\|_{L^2_{\widehat{x}_m}} \|\Xi_m(D)P_{k'} w\|_{L^2_{\widehat{x}_m}}\Big\|_{L^2_{t,x_m}} \notag\\
	&\quad + \int_0^s \Big\|\|P_{k} e^{(s-r)\Delta_x}K(\bm{\psi})(r)^y\|_{L^2_{\widehat{x}_m}} \|\Xi_m(D)P_{k'} w\|_{L^2_{\widehat{x}_m}}\Big\|_{L^2_{t,x_m}} \notag\\
	&\lesssim \langle 2^{2k}s \rangle^{-N} \sup_y \Big\|\|P_{k}  \bm{\psi}(0)^y\|_{L^2_{\widehat{x}_m}} \|\Xi_m(D)P_{k'} w\|_{L^2_{\widehat{x}_m}}\Big\|_{L^2_{t,x_m}}\notag\\
	&\quad +  \int_0^s \langle 2^{2k}(s-r) \rangle^{-N} \sup_y  \Big\|\|P_{k}  K(\bm{\psi})(r)^y\|_{L^2_{\widehat{x}_m}}\|\Xi_m(D)P_{k'} w\|_{L^2_{\widehat{x}_m}}\Big\|_{L^2_{t,x_m}}.\label{K-integral}
\end{align}
The first term is controlled by bootstrap assumptions and Proposition \ref{bilinear1}. For the second, decompose $K(\bm{\psi})$ as
\begin{equation}
  P_k K(\bm{\psi}) = \sum_{k_1,k_2} P_k L(P_{k_1}Q,P_{k_2}\bm{\psi}),\ \text{where}\ Q = (2^{k}\bm{A},\pa_x \bm{A}, \bm{A}\cdot \bm{A}, \bm{\psi}\cdot\bm{\psi}). 
\end{equation}

Using \eqref{A-bound}, \eqref{b1}, \eqref{b4} and $\sum a_k^2\le\varepsilon$, we have 
\begin{equation}
  \|P_{k_1}Q(s)\|_{G}\lesssim \varepsilon^{\frac{1}{2}}2^{\max\{k,k_1\}}2^{-\sigma k_1} (2^{2k_1}s)^{-\frac{1}{8}}\langle 2^{2k_1}s \rangle^{-3} b_{k_1}(\sigma).
\end{equation}

Denote by $k_{\min}, k_{\text{med}}$ and $k_{\max}$ the minimum, median and maximum of the tuple $(k_1,k_2,k)$ respectively. For summation over $|k_{\text{med}}-k_{\max}|\le 4$ and $k'\ge  k_2 + 100$,  we have
\begin{align}
	&\eqref{K-integral}\lesssim \sum_{k_1,k_2}2^{\frac{k_{\min}}{2}}\int_0^s \langle 2^{2k}(s-r) \rangle^{-N}\notag\\
	&\qquad\qquad\qquad\  \|P_{k_1}Q(r)\|_{L^\infty_{t,x_m}L^2_{\widehat{x}_m}}\sup_y  \Big\|\| P_{k_2}\bm{\psi}(r)^y\|_{L^2_{\widehat{x}_m}}\|\Xi_m(D)P_{k'} w\|_{L^2_{\widehat{x}_m}}\Big\|_{L^2_{t,x_m}}\mathrm{d}r \notag\\
	&\lesssim B_3 2^{-\frac{k'}{2}}\mathcal{D}(P_{k'}w)\sum_{k_1,k_2}2^{\max\{k,k_1\}}2^{\frac{k_1 + k_{\min}}{2}}2^{-\sigma \max\{k_1,k_2\}}b_{\min\{k_1,k_2\}} b_{\max\{k_1,k_2\}}(\sigma)\notag\\
	&\qquad\  \int_{0}^s \langle 2^{2k}(s-r) \rangle^{-N} (2^{2k_1}r)^{-\frac{1}{8}}\langle 2^{2k_1}r \rangle^{-3}\langle 2^{2k_2}r \rangle^{-3}\mathrm{d}r\notag\\
	&\lesssim B_3 2^{-\sigma k }2^{-\frac{k'}{2}}\mathcal{D}(P_{k'}w)\sum_{k_1,k_2}2^{\max\{k,k_1\}}2^{\frac{k_1 + k_{\min}}{2}}b_{\min\{k_1,k_2\}}b_{\max\{k_1,k_2\}}(\sigma)\notag\\
	&\qquad \langle 2^{2k}s \rangle^{-3}\int_{0}^s \langle 2^{2k}(s-r) \rangle^{3-N} (2^{2k_1}r)^{-\frac{1}{8}}\langle 2^{2\min\{k_1,k_2\}}r \rangle^{-3} \mathrm{d}r\notag\\
	&\lesssim B_3 2^{-\frac{k'}{2}} 2^{-\sigma k }\langle 2^{2k}s \rangle^{-3}\mathcal{D}(P_{k'}w) \sum_{k_1,k_2} 2^{\frac{k_{\min}-k_{\max}}{2}}b_{\min\{k_1,k_2\}}b_{\max\{k_1,k_2\}}(\sigma) \notag\\
	&\lesssim B_3 2^{-\frac{k'}{2}} 2^{-\sigma k }\langle 2^{2k}s \rangle^{-3}\mathcal{D}(P_{k'}w) b_k b_k(\sigma)\lesssim B_3 \varepsilon^{\frac{1}{2}}2^{-\frac{k'}{2}} 2^{-\sigma k }\langle 2^{2k}s \rangle^{-3}\mathcal{D}(P_{k'}w)  b_k(\sigma),\label{com1}
\end{align}
where we use  the fact that $\max\{k_1,k_2\} \ge  k-4$ and $\max\{\min\{k_1,k_2\},k\}\ge  k_{\max}-4$. 

For summation over $|k_{\text{med}}-k_{\max}|\le 4$ and $k_2\ge k'- 100\ge k$, we have $|k_1-k_2|\le 10$ and thus  
\begin{align}
	&\eqref{K-integral}\lesssim \sum_{k_1,k_2 } \int_0^s \langle 2^{2k}(s-r) \rangle^{-N}\notag\\
 &\qquad\qquad\qquad\  \|\Xi_m(D)P_{k'} w\|_{L^\infty_{t}L^2_x}  2^{\frac{k}{2}} \| P_{k_2}\bm{\psi}(r) \|_{G}\|P_{k_1}Q(r)\|_{G} \mathrm{d}r \notag\\
	&\lesssim \sum_{k_2 \ge  k'-100} \int_0^s \langle 2^{2k}(s-r) \rangle^{-N} 2^{\frac{k}{2}+ k_2}2^{-\sigma k_2} \mathcal{D}(P_{k'}w)b_{k_2}b_{k_2}(\sigma) (2^{2k_2}r)^{-\frac{1}{8}}\langle 2^{2k_2}r \rangle^{-6} \mathrm{d} r \notag\\
	&\lesssim  \langle 2^{2k}s \rangle^{-3} 2^{-\frac{k'}{2}}2^{-\sigma k'} \mathcal{D}(P_{k'}w)2^{\frac{k-k'}{2}}b_{k'}b_{k'}(\sigma) \notag\\
	&\lesssim \langle 2^{2k}s \rangle^{-3} 2^{-\frac{k'}{2}}2^{-\sigma k'} \mathcal{D}(P_{k'}w) b_{k}(\sigma).
\end{align}	
Combining this with \eqref{K-integral} and \eqref{com1} gives $B_3\lesssim 1+ \varepsilon^{\frac{1}{2}}B_3$ and hence $B_3\lesssim 1$. This gives the bound \eqref{psis-bilinear1}.

Next we prove \eqref{psis-bilinear2}. We again use the formula \eqref{heatformula} and repeat the arguments in \eqref{K-integral}. It remains to control
\begin{align}
	&\int_0^s \langle 2^{2k}(s-r) \rangle^{-N}\sup_{y}\|P_{k}K(\psi)(r)P_{k'}w^{y}\|_{L^2_{t,x}}\mathrm{d}r\notag\\
	&\lesssim \|P_{k'}w\|_{G}\sum_{k_1,k_2}\int_0^s  \langle 2^{2k}(s-r) \rangle^{-N} 2^{\min\{k',k_{\min}\}}\sup_{y}\|P_k L(P_{k_1}Q(r),P_{k_2}\bm{\psi}(r)) \|_{L^2_{t,x}}\mathrm{d}r\notag\\
	&\lesssim  \mathcal{D}(P_{k'}w)\sum_{k_1,k_2} 2^{\min\{k',k_{\min}\}}2^{\max\{k_1,k\}}\notag\\
	&\qquad\  \int_{0}^s \langle 2^{2k}(s-r) \rangle^{-N} (2^{2k_1}r)^{-\frac{1}{8}}\langle 2^{2k_1}r \rangle^{-3}\langle 2^{2k_2}r \rangle^{-4} 2^{-\sigma \max\{k_1,k_2\}}b_{\min\{k_1,k_2\}}b_{\max\{k_1,k_2\}}(\sigma)\mathrm{d}r\notag\\
	&\lesssim  \mathcal{D}(P_{k'}w)2^{-\sigma k} \langle 2^{2k}s \rangle^{-3}\sum_{k_1,k_2} 2^{\min\{k',k_{\min}\} - k_{\max}} b_{\min\{k_1,k_2\}}b_{\max\{k_1,k_2\}}(\sigma)\notag\\
	&\lesssim  \mathcal{D}(P_{k'}w)2^{-\sigma k} 2^{\frac{k'-k}{2}}\langle 2^{2k}s \rangle^{-3}b_{k}b_k(\sigma),\notag
\end{align} 
where we bound the integral as in \eqref{com1}. This gives \eqref{psis-bilinear2}.

It follows from \eqref{psis-bilinear1} and \eqref{psis-bilinear2} that
	\begin{align} 
	 &\sup_y \|P_{k}\bm{\psi}(s) P_{k'}w^y\|_{L^2_{t,x}}\lesssim 2^{-\frac{|k -k'|}{2}} 2^{-\sigma k} \langle 2^{2k}s \rangle^{-3} b_k(\sigma)\mathcal{D}(P_{k'}w),\\
	 & \sup_y \|P_{k}\pa_x \bm{\psi}(s) P_{k'}w^y\|_{L^2_{t,x}} \lesssim 2^{-\frac{|k -k'|}{2}} 2^{-(\sigma-1)k}  \langle 2^{2k}s \rangle^{-3}b_{k}(\sigma)\mathcal{D}(P_{k'}w), 
	\end{align}
which gives the bound \eqref{psis-main} and  \eqref{psis-main-2} once we prove that 
\begin{equation}
  \sup_y \|P_{k}(\bm{A}(s) \bm{\psi}(s)) P_{k'}w^y\|_{L^2_{t,x}} \lesssim 2^{-\frac{|k -k'|}{2}} 2^{-(\sigma-1)k} (2^{2k}s)^{-\frac{1}{4}}\langle 2^{2k}s \rangle^{-3}b_{k}(\sigma)\mathcal{D}(P_{k'}w).
\end{equation}

If $k\le  k'$, we apply Bony calculus to the product $P_{k}(\bm{A}(s) \bm{\psi}(s))$ and obtain
\begin{align}
	&\sup_y \|P_{k}(\bm{A}(s) \bm{\psi}(s)) P_{k'}w^y\|_{L^2_{t,x}}\label{App}\\
	& \lesssim \sum_{k_1,k_2,m} \big\| \|P_{k}(P_{k_1}\bm{A}(s) P_{k_2}\bm{\psi}(s)) \|_{L^\infty_{\hat{x}_m}} \|\Xi_m(D)P_{k'}w^y\|_{L^2_{\hat{x}_m}} \big\|_{L^2_{t,x_m}} \notag \\
	&\lesssim \sum_{k_1,k_2,m}  2^{\frac{1}{2}(\min\{k,k_1,k_2\} + k)} \|P_{k_1}\bm{A}(s) \|_{L^\infty_{t,x_m}L^2_{\hat{x}_m}} \big\| \| P_{k_2}\bm{\psi}(s)  \|_{L^2_{\hat{x}_m}} \|\Xi_m(D)P_{k'}w^y\|_{L^2_{\hat{x}_m}} \big\|_{L^2_{t,x_m}}  \notag \\
	&\lesssim 2^{-\frac{k'}{2}} \sum_{k_1,k_2 }  2^{\frac{1}{2}(\min\{k,k_1,k_2\} + k)} 2^{\frac{k_1}{2}} 2^{-\sigma \max\{k_1,k_2\}}b_{\min\{k_1,k_2\}}b_{\max\{k_1,k_2\}}(\sigma)\mathcal{D}(P_{k'}w)\notag\\
	&\qquad \qquad \qquad (2^{2k_1}s)^{-\frac{1}{8}} \langle 2^{2k_1} s \rangle^{-\frac{13}{4}}\langle 2^{2k_2} s \rangle^{-4}, \notag
                                               \end{align}
where we use the following bound for $\bm{A}(s)$ derived from \eqref{A-bound} and \eqref{up}:
\begin{equation}
  \|P_{k_1}\bm{A}(s)\|_G\lesssim 2^{-\sigma k_1} (2^{2k_1}s)^{-\frac{1}{8}} \langle 2^{2k_1} s \rangle^{-\frac{13}{4}} b_{k_1}(\sigma).
\end{equation}

For the High-Low and Low-High interactions, we have 
\begin{align}
	&\eqref{App} \lesssim 2^{\frac{k-k'}{2}} 2^{-\sigma k}b_k(\sigma) 2^{\frac{k}{4}}s^{-\frac{1}{8}}\sum_{k_3\le  k} 2^{\frac{ k_3}{2}}   \langle 2^{2k_3} s \rangle^{-\frac{13}{4}}\langle 2^{2k} s \rangle^{-\frac{13}{4}} \mathcal{D}(P_{k'}w)\notag\\
	& \lesssim 2^{\frac{k-k'}{2}} 2^{-\sigma k}b_k(\sigma) 2^{\frac{3k}{4}}s^{-\frac{1}{8}}\langle 2^{2k} s \rangle^{-\frac{13}{4}} \mathcal{D}(P_{k'}w)\notag\\
	&\lesssim 2^{\frac{k-k'}{2}} 2^{-(\sigma - 1)k}b_k(\sigma) (2^{2k} s)^{-\frac{1}{8}} \langle 2^{2k} s \rangle^{-\frac{13}{4}}\mathcal{D}(P_{k'}w),   \notag
\end{align}
and for High-High interactions we have 
\begin{align}
		&\eqref{App} \lesssim 2^{\frac{k-k'}{2}} 2^{-\sigma k} 2^{\frac{k}{2}}s^{-\frac{1}{4}}\sum_{k_3\ge  k} (2^{2 k_3}s)^{\frac{1}{8}}   \langle 2^{2k_3} s \rangle^{-6} b_{k_3}(\sigma)\mathcal{D}(P_{k'}w)\notag\\
	&\lesssim 2^{\frac{k-k'}{2}} 2^{-(\sigma - 1)k}b_k(\sigma) (2^{2k} s)^{-\frac{1}{4}} \langle 2^{2k} s \rangle^{-3}\mathcal{D}(P_{k'}w).   \notag
\end{align}
Thus the bound for $k'\ge  k$ is consistent with the stated bound \eqref{psis-main-2}. 

If $k'\le  k$, we have 
\begin{align}
  &\sup_y \|P_{k}(\bm{A}(s) \bm{\psi}(s)) P_{k'}w^y\|_{L^2_{t,x}} \\
	& \lesssim \sum_{k_1,k_2 } 2^{\min\{k',k_1,k_2\}} \|P_{k_1}\bm{A}(s)\|_G \|P_{k_2} \bm{\psi}(s)\|_{G} \| P_{k'}w\|_{G} \notag\\
	&\lesssim \mathcal{D}(P_{k'}w) \sum_{k_1,k_2} 2^{\min\{k',k_1,k_2\}} 2^{-\sigma \max\{k_1,k_2\}}b_{\min\{k_1,k_2\}}b_{\max\{k_1,k_2\}}(\sigma) (2^{2k_1}s)^{-\frac{1}{8}} \langle 2^{2k_1} s \rangle^{-\frac{13}{4}}\langle 2^{2k_2} s \rangle^{-4} \notag \\
	&\lesssim 2^{-\sigma k } 2^{\frac{k'}{2}}  s^{-\frac{1}{4}}\mathcal{D}(P_{k'}w) \sum_{k_1,k_2} 2^{\frac{1}{8}(\min\{k,k_1,k_2\}-\max\{k,k_1,k_2\})}b_{\min\{k_1,k_2\}}b_{\max\{k_1,k_2\}}(\sigma) \notag\\
	&\qquad \qquad (2^{2k_1}s)^{\frac{1}{16}} (2^{2k_2}s)^{\frac{1}{16}}\langle 2^{2k_1} s \rangle^{-\frac{13}{4}}\langle 2^{2k_2} s \rangle^{-4} \notag\\
	&\lesssim 2^{-(\sigma-1)k} 2^{\frac{k'-k}{2}} (2^{2k}s)^{-\frac{1}{4}} \langle 2^{2k} s \rangle^{-3} b_{k}b_{k}(\sigma) \mathcal{D}(P_{k'}w), \notag
\end{align}
where we use the fact that $\max\{k_1,k_2\}\ge  k -4 $. This gives the stated bound when $k'\le  k$ and completes the proof.

\end{proof}

\section{Estimates for Ishimori equation}
We begin by recalling the nonlinear terms in the Ishimori system. Denote by 
\begin{align}
 \mathcal{N}_m & =  -2i \mu_l A_l \pa_l \psi_m +(A_t + \mu_l(A_l^2-i\pa_l A_l))\psi_m - i \mu_l\psi_l\Im (\psi_m\bar{\psi}_{l})  \notag\\
	 & \quad+  2i \epsilon_{ij}\kappa(D_m\psi_l R_l  R_i A_j + \psi_l \pa_m(R_l  R_i A_j))  \notag\\
              & =  - i \mu_l A_l \pa_l \psi_m  -i\mu_l\pa_l (A_l\psi_m) + 2i \epsilon_{ij}\kappa\ \pa_m( \psi_l  R_l R_i  A_j)\notag\\
			  &\quad\  + (A_t + \mu_l A_l^2)\psi_m  - 2 \epsilon_{ij}\kappa\ \psi_l   A_m R_l R_i A_j  - i \mu_l\psi_l\Im (\psi_m\bar{\psi}_{l})\notag\\
			  & =: \mathcal{N}_{m,1} + \mathcal{N}_{m,2}
\end{align}
 appearing in the Ishimori equation \eqref{psi}.  

\begin{lemma}
	Under the bootstrap assumptions in Proposition \ref{SchrodingerBA}, we have 
	\begin{align}
	&\sup_y  \|P_k( \bm{A},\bm{\psi} ) P_{k'}w^y \|_{L^2_{t,x}}\lesssim 2^{-\sigma k }2^{-\frac{|k-k'|}{2}}\mathcal{D}(P_{k'}w) b_{k}(\sigma),\label{Apsi-1}\\
	&\sup_{y}  \|P_{k }w^y P_{k}[(\bm{A}\cdot\bm{A},\bm{\psi}\cdot\bm{\psi}) \bm{\psi}]     \|_{L^1_{t,x}}\lesssim \varepsilon  2^{-\sigma k } \mathcal{D}(P_{k }w) b_k(\sigma),\label{A2psi}\\
	&\sup_{y}  \|P_{k }w^y P_{k}(A_t \bm{\psi})     \|_{L^1_{t,x}}\lesssim \varepsilon  2^{-\sigma k } \mathcal{D}(P_{k }w) b_k(\sigma). \label{Atpsi}
	\end{align}

	Using these bounds, we have 
	 \begin{equation}\label{N2}
   \sup_{y}\|P_{k}w^y P_k \mathcal{N}_{2,m}\|\lesssim \varepsilon2^{-\sigma k }b_k(\sigma)\mathcal{D}(P_k w).	
 \end{equation}
\end{lemma}
\begin{proof}
	We begin with the proof of \eqref{Apsi-1}. 
	The bound for the $\bm{\psi}$ term
	\begin{equation}
	   \|P_k \bm{\psi}  P_{k'}w^y \|_{L^2_{t,x}}\lesssim 2^{-\sigma k }2^{-\frac{|k-k'|}{2}}\mathcal{D}(P_{k'}w) b_{k}(\sigma)
	\end{equation}
	follows directly from the bootstrap assumptions \eqref{SchrodingerBA} and the bilinear estimate \eqref{bilinear2}. 
	Thus it remains to estimate the product involving the connection coefficient $\bm{A}$. Using  the integral representation \eqref{Aid} for $\bm{A}$, we  need to control 
	\begin{equation}
	  \sup_y \|P_k \bm{A}   P_{k'}w^y \|_{L^2_{t,x}}\lesssim\sum_{k_1,k_2}\int_0^\infty \|P_{k}(P_{k_1}\bm{\psi}(s) P_{k_2}[D_x\bm{\psi}(s)]) P_{k'}w^y \|_{L^2_{t,x}} \mathrm{d}s.
	\end{equation}

    We consider two cases based on the relative sizes of $k$ and $k'$.

	If $k\le  k'$, for the Low-High and  High-Low interactions, where  $\max\{k_1,k_2\}\le k +4$, and for the High-High interactions in the regime $k_2\le \frac{k+k'}{2}$, we use the bilinear estimate \eqref{psis-main-2} to obtain
	\begin{align}
	 & \sum_{k_1,k_2 }\int_0^\infty \|P_{k}(P_{k_1}\bm{\psi}(s) P_{k_2}[D_x\bm{\psi}(s)]) P_{k'}w^y \|_{L^2_{t,x}} \mathrm{d}s\notag\\ 
	 &\lesssim  \sum_{k_1,k_2 }\int_0^\infty \|P_{k_1} \bm{\psi}(s)\|_{L^\infty_{t,x}} \sup_y \|  P_{k_2} (D_x\bm{\psi}(s))  P_{k'}w^y \|_{L^2_{t,x}} \mathrm{d}s \notag\\
	 &\lesssim \sum_{k_1,k_2 }2^{k_1 + k_2} 2^{-\frac{|k_2-k'|}{2}}\mathcal{D}(P_{k'}w)\notag\\
	 &\quad\quad\int_0^\infty(2^{2k_2} s)^{-\frac{1}{4}} \langle 2^{2k_1} s \rangle^{-3}\langle 2^{2k_2} s \rangle^{-3}2^{-\sigma\max\{k_1,k_2\}}b_{\min\{k_1,k_2\}}b_{\max\{k_1,k_2\}}(\sigma)\mathrm{d}s \notag\\
	 &\lesssim 2^{-\sigma k }\mathcal{D}(P_{k'}w) \sum_{k_1,k_2 } 2^{-\frac{|k_1-k_2|}{2}} 2^{-\frac{|k_2-k'|}{2}}  b_{\min\{k_1,k_2\}}b_{\max\{k_1,k_2\}}(\sigma)\notag\\
	 &\lesssim 2^{-\sigma k }\mathcal{D}(P_{k'}w) \sum_{k_3\le  k}2^{-\frac{|k_3-k|}{2}}\max\{2^{-\frac{|k-k'|}{2}},2^{-\frac{|k_3-k'|}{2}}\}b_{k_3}b_k(\sigma)\notag\\
	 &\quad \quad + 2^{-\sigma k }\mathcal{D}(P_{k'}w) \sum_{ k-4\le k_3\le \frac{k+k'}{2}} 2^{-\frac{|k_3-k'|}{2}} b_{k_3}b_{k_3}(\sigma)\notag\\
	 &\lesssim  \varepsilon^{\frac{1}{2}}2^{-\sigma k }2^{-\frac{|k-k'|}{2}}\mathcal{D}(P_{k'}w) b_{k}(\sigma),
	\end{align}

	while for the High-High interactions with  $|k_1-k_2|\le  4, k_1\ge k-4$ in the regime $k_2\ge \frac{k+k'}{2}$, we have 
	\begin{align}
	 & \sum_{k_1,k_2 }\int_0^\infty \|P_{k}(P_{k_1}\bm{\psi}(s) P_{k_2}[D_x\bm{\psi}(s)]) P_{k'}w^y \|_{L^2_{t,x}} \mathrm{d}s\notag\\ 
	 &\lesssim  \sum_{k_1,k_2 }\int_0^\infty \| P_{k'}w \|_{L^\infty_{t }L^2_x}   \| P_{k}(P_{k_1}\bm{\psi}(s) P_{k_2}[D_x\bm{\psi}(s)])\|_{L^2_{t}L^\infty_x} \mathrm{d}s \notag\\
	 &\lesssim \sum_{k_1,k_2 }2^{k} \|P_{k_1}\bm{\psi}(s) P_{k_2}[D_x\bm{\psi}(s)]\|_{L^2_{t,x}}  \mathcal{D}(P_{k'}w)\notag\\
	 &\lesssim \sum_{k_3\ge \frac{1}{2}(k+k')}2^k  \mathcal{D}(P_{k'}w)  \int_0^\infty(2^{2k_3} s)^{-\frac{1}{4}}  \langle 2^{2k_3} s \rangle^{-6}2^{-\sigma k_3}b_{k_3}b_{k_3}(\sigma)\mathrm{d}s \notag\\
	 &\lesssim 2^{-\sigma k }\mathcal{D}(P_{k'}w) \sum_{k_3\ge \frac{1}{2}(k+k')}2^{k-k_3}b_{k_3}b_{k_3}(\sigma)\notag\\
	 &\lesssim  \varepsilon^{\frac{1}{2}}2^{-\sigma k }2^{-\frac{|k-k'|}{2}}\mathcal{D}(P_{k'}w) b_{k}(\sigma),
	\end{align}
Combining the above subcases yields the desired bound \eqref{Apsi-1} when $k'\ge k$.

If $k'\le k$, for the  Low-High interaction with    $k_1 \le k-4$, $|k_2-k|\le 8$ , we have
\begin{align}
	 & \sum^{|k_2-k|\le 8}_{k_1\le k-4}\int_0^\infty \sup_y\|P_{k}(P_{k_1}\bm{\psi}(s) P_{k_2}[D_x\bm{\psi}(s)]) P_{k'}w^y \|_{L^2_{t,x}} \notag\\
	 &\lesssim    \sum^{|k_2-k|\le 8}_{k_1\le k-4}\int_0^\infty \|P_{k_1} \bm{\psi}(s)\|_{L^\infty_{t,x}} \sup_y \|  P_{k_2} (D_x\bm{\psi}(s))  P_{k'}w^y \|_{L^2_{t,x}} \mathrm{d}s \notag\\
	 &\lesssim \sum^{|k_2-k|\le 8}_{k_1\le  k-4}2^{k_1+k_2}2^{-\frac{|k_2-k'|}{2}} 2^{-\sigma k_2 }b_{k_1}b_{k_2}(\sigma)\mathcal{D}(P_{k'}w)\int_0^\infty(2^{2k_2} s)^{-\frac{1}{4}} \langle 2^{2k_1} s \rangle^{-3}\langle 2^{2k_2} s \rangle^{-3} \mathrm{d}s \notag\\
	 &\lesssim  2^{-\sigma k  }2^{-\frac{|k-k'|}{2}}b_{k }(\sigma)\mathcal{D}(P_{k'}w) \sum_{k_1\le 4} 2^{-\frac{|k_1-k|}{2}} b_{k_1}\lesssim \varepsilon^{\frac{1}{2}}2^{-\sigma k  }2^{-\frac{|k-k'|}{2}} b_{k }(\sigma)\mathcal{D}(P_{k'}w).\notag
\end{align}

For the High-Low and High-High interactions, we place the high-frequency term in $L_{t,x}^2$ and use the pointwise bound for $P_{k'}w$ to obtain
	\begin{align}
	 & \sum_{k_1\ge  k-4}^{k_2}\int_0^\infty \sup_y\|P_{k}(P_{k_1}\bm{\psi}(s) P_{k_2}[D_x\bm{\psi}(s)]) P_{k'}w^y \|_{L^2_{t,x}} \mathrm{d}s\notag\\ 
	 &\lesssim  \sum_{k_1\ge  k-4}^{k_2}\int_0^\infty \|P_{k'}w\|_{L^\infty_{t,x}}  \| P_{k_1} \bm{\psi}(s) P_{k_2} (D_x\bm{\psi}(s))   \|_{L^2_{t,x}} \mathrm{d}s \notag\\
	 &\lesssim \sum_{k_1\ge  k-4}^{k_2} 2^{k' + k_2} \mathcal{D}(P_{k'}w)\notag\\
	 &\quad\quad\int_0^\infty(2^{2k_2} s)^{-\frac{1}{4}} \langle 2^{2k_1} s \rangle^{-3}\langle 2^{2k_2} s \rangle^{-3}2^{-\sigma\max\{k_1,k_2\}}b_{\min\{k_1,k_2\}}b_{\max\{k_1,k_2\}}(\sigma)\mathrm{d}s \notag\\
	 &\lesssim 2^{-\sigma k }2^{-\frac{k-k'}{2}}\mathcal{D}(P_{k'}w) \sum_{k_1\ge  k-4}^{k_2}2^{-\frac{|k_1-k_2|}{2}} 2^{-\frac{|k_1-k|}{2}}  b_{\min\{k_1,k_2\}}b_{\max\{k_1,k_2\}}(\sigma),\notag\\
	 &\lesssim  \varepsilon^{\frac{1}{2}} 2^{-\sigma k }2^{-\frac{|k-k'|}{2}}\mathcal{D}(P_{k'}w)b_{k}(\sigma),
	\end{align}
	where we use the fact that $|k_2-\min\{k_1,k_2\}|\le 10$ in the High-Low and High-High regime. Combining both subcases gives the bound \eqref{Apsi-1} when $k' \le k$.

We next prove \eqref{A2psi}. We focus on the representative term $P_{k }w^y P_{k}(\bm{A}\cdot \bm{A} \cdot \bm{\psi})$ in \eqref{A2psi}, as estimates for other terms are similar. Applying Bony's paraproduct decomposition, we obtain
\begin{align}
	&\sup_{y}  \|P_{k }w^y P_{k} (\bm{A}\cdot\bm{A}\cdot \bm{\psi})     \|_{L^1_{t,x}}\lesssim \sum_{k_i}\sup_{y}  \|P_{k }w^y P_{k} (P_{k_1}\bm{A}\cdot P_{k_2}\bm{A}\cdot P_{k_3}\bm{\psi})     \|_{L^1_{t,x}}\notag\\
	&\lesssim \sum_{k_i}\min \{2^{-\frac{1}{2}(|k-k_1|+ |k_2-k_3|) },2^{-\frac{1}{2}(|k-k_2|+ |k_1-k_3|)}\}2^{-\sigma k_{\max}}b_{k_{\min}}b_{k_{\text{med}}}b_{k_{\max}}(\sigma)\mathcal{D}(P_k w)\notag\\
	&\lesssim 2^{-\sigma k }\mathcal{D}(P_k w) \sum_{k_i} 2^{-\frac{1}{8}(|k-k_1|+ |k_1-k_2| + |k_2-k_3|)}b_{k_{\min}}b_{k_{\text{med}}}b_{k_{\max}}(\sigma)\notag\\
	&\lesssim \varepsilon 2^{-\sigma k }\mathcal{D}(P_k w) b_k(\sigma),
\end{align}
where  $(k_{\min},k_{\text{med}},k_{\max})$ is the increasing rearrangement of $(k_1,k_2,k_3)$ and the off-diagonal factor comes from applying \eqref{Apsi-1} to the pairs $(P_k w^y, P_{k_1}\bm{A})$ and $(P_{k_2}\bm{A}, P_{k_3}\bm{\psi})$ in the two possible bilinear groupings.

Finally, we prove  \eqref{Atpsi}. Using the identity $A_t = -\int_{s}^{\infty}\Im(\overline{\psi_t}D_l \psi_l )(r)\mathrm{d}r$, we need to distinguish between two regularity regimes for $\psi_t$, as given by \eqref{psitk}. 

For $\sigma\le  \frac{1}{5}$, we shall use the bound $\|P_{k}\psi_t(s)\|\lesssim \varepsilon^{\frac{1}{2}}2^{k}\langle 2^{2k}s \rangle^{-4}$ to obtain 
\begin{align}
	&\|P_{k }w^y P_{k}(A_t \bm{\psi}) \|_{L^1_{t,x}} \lesssim \sum_{k_i} \int_0^{+\infty}\|P_{k }w^y P_{k}(P_{k_1}\psi_t(s) P_{k_2}(D_x \bm{\psi}(s)) P_{k_3}\bm{\psi})  \|_{L^1_{t,x}} \mathrm{d}s\notag\\
    &\lesssim \sum_{k_i}2^{k_1 + k_2} \min\{2^{-\frac{|k_2-k_3|}{2}},2^{-\frac{|k-k_2|}{2}}\} 2^{-\sigma k_3}\varepsilon^{\frac{1}{2}}b_{k_2}b_{k_3}(\sigma)\mathcal{D}(P_k w)\notag\\
	&\qquad\qquad    \int_0^\infty  ( 2^{2k_2}s  )^{-\frac{1}{4}} \langle 2^{2k_1}s \rangle^{-3}\langle 2^{2k_2}s \rangle^{-3}  \mathrm{d}s\notag \\
	&\lesssim \varepsilon^{\frac{1}{2}} 2^{-\sigma k} b_{k}(\sigma) \mathcal{D}(P_k w) \sum_{k_i} 2^{-\frac{|k_1-k_2| }{2}} 2^{(\sigma+\delta)|k-k_3|}\min\{2^{-\frac{|k_2-k_3|}{2}},2^{-\frac{|k-k_2|}{2}}\}b_{k_2}\notag\\
	&\lesssim \varepsilon  2^{-\sigma k} b_{k}(\sigma)\mathcal{D}(P_k w),
\end{align}
where the factor $\min\{2^{-\frac{|k_2-k_3|}{2}},2^{-\frac{|k-k_2|}{2}}\}$ comes from applying \eqref{psis-main-2} either to the pair $(P_k w^y, P_{k_2}(D_x \bm{\psi}))$ or to the pair $(P_{k_3} \bm{\psi}, P_{k_2}(D_x \bm{\psi}))$, whichever yields better decay.

For $\sigma\ge  \frac{1}{5}$, we shall use the bound $\|P_{k}\psi_t(s)\|\lesssim  2^{-(\sigma-1)k}\langle 2^{2k}s \rangle^{-4} b_k(\sigma)$ to obtain
\begin{align}
	&\|P_{k }w^y P_{k}(A_t \bm{\psi}) \|_{L^1_{t,x}} \lesssim \sum_{k_i} \int_0^{+\infty}\|P_{k }w^y P_{k}(P_{k_1}\psi_t(s) P_{k_2}(D_x \bm{\psi}(s)) P_{k_3}\bm{\psi})  \|_{L^1_{t,x}} \mathrm{d}s\notag\\
    &\lesssim \sum_{k_i}2^{k_1 + k_2} \min\{2^{-\frac{|k_2-k_3|}{2}},2^{-\frac{|k-k_2|}{2}}\} 2^{-\sigma k_{\max}} b_{k_{\min}}b_{k_{\text{med}}}b_{k_{\max}}(\sigma)\mathcal{D}(P_k w)\notag\\
	&\qquad\qquad    \int_0^\infty  ( 2^{2k_2}s  )^{-\frac{1}{4}} \langle 2^{2k_1}s \rangle^{-3}\langle 2^{2k_2}s \rangle^{-3}  \mathrm{d}s\notag \\
	&\lesssim  2^{-\sigma k} \mathcal{D}(P_k w) \sum_{k_i} 2^{-\frac{|k_1-k_2|}{2}} \min\{2^{-\frac{|k_2-k_3|}{2}},2^{-\frac{|k-k_2|}{2}}\} b_{k_{\min}}b_{k_{\text{med}}}b_{k_{\max}}(\sigma)\notag\\
	&\lesssim \varepsilon  2^{-\sigma k} b_{k}(\sigma)\mathcal{D}(P_k w).
\end{align}
This completes the proof of  \eqref{Atpsi}.

The estimate \eqref{N2} for $\mathcal{N}_{2,m}$ follows directly from \eqref{Atpsi} and \eqref{A2psi}, combined with the fact that the remaining term $P_k w^y P_k(\psi_l   A_m R_l R_i A_j)$ can be controlled in a similar manner to \eqref{A2psi}, since 
\begin{equation}
  P_k w^y P_k(\psi_l   A_m R_l R_i A_j) = \sum_{k_i}P_k w^y L(P_{k_1}\bm{\psi},   P_{k_2}\bm{A},  P_{k_3} \bm{A})
\end{equation}  
and \eqref{Apsi-1} applies well to the $L$-notation product terms. This completes the proof.
\end{proof}

\begin{lemma}
	Under the bootstrap assumptions, the derivative part of the nonlinearity satisfies
\begin{equation}\label{N1}
  \sup_{y}\|P_{k}w^y P_k \mathcal{N}_{1,m}\|\lesssim \varepsilon 2^{-\sigma k }b_k(\sigma)\mathcal{D}(P_k w)
\end{equation}
\end{lemma}
\begin{proof}
	We apply Bony's paraproduct decomposition to the terms in $\mathcal{N}_{1,m}$
  \begin{align}
	 &\sup_{y}\|P_{k}w^y P_k \mathcal{N}_{1,m}\|_{L^1_{t,x}}\notag\\
	 &\lesssim \sum_{k',k_3}\int_{0}^\infty\sup_{y}\|P_{k}w^y \cdot P_k(\pa_x L(P_{k'}A P_{k_3} \bm{\psi}), L(P_{k'}A \pa_x P_{k_3} \bm{\psi}))\|_{L^1_{t,x}}\mathrm{d}s\notag\\
	 &\lesssim\sum_{k_i}2^{\max\{k,k_3\}}\int_{0}^\infty\sup_{y,y'}\|P_{k}w^y P_{k_1}\bm{\psi}(s)P_{k_2}D_x \bm{\psi} P_{k_3} \bm{\psi}^{y'}  \| _{L^1_{t,x}}\notag\\
	 &\lesssim \sum_{k_i} 2^{\max\{k,k_3\}+ k_2}\min\{2^{-\frac{1}{2}(|k-k_1|+|k_2-k_3|)},2^{-\frac{1}{2}(|k-k_2|+|k_1-k_3|)}\}\notag\\
	 &\quad\cdot \int_0^\infty (2^{2k_2}s)^{-\frac{1}{4}} \langle 2^{2k_1} s \rangle^{-3}\langle 2^{2k_2} s \rangle^{-3}\mathrm{d}s\cdot 2^{-\sigma k_{\max}} b_{k_{\min}}b_{k_{\text{med}}}b_{k_{\max}}(\sigma)\mathcal{D}(P_k w)\notag\\
	 &\lesssim \sum_{k_i} 2^{\max\{k,k_3\}+ \frac{1}{2}k_2-\frac{3}{2}\max\{k_1,k_2\}}\min\{2^{-\frac{1}{2}(|k-k_1|+|k_2-k_3|)},2^{-\frac{1}{2}(|k-k_2|+|k_1-k_3|)}\}\notag\\
	 &\quad\     b_{k_{\min}}b_{k_{\text{med}}}b_{k_{\max}}(\sigma)\mathcal{D}(P_k w)\notag\\
	 &\lesssim 2^{-\sigma k}\mathcal{D}(P_k w)\sum_{k_i}C(k,k_1,k_2,k_3)b_{k_{\min}}b_{k_{\text{med}}}b_{k_{\max}}(\sigma),\label{DApsi-1}
  \end{align}
  where $(k_{\min},k_{\text{med}},k_{\max})$ is the increasing rearrangement of  $(k_1,k_2,k_3)$, and the factor $C(k,k_1,k_2,k_3)$ is defined as 
  \begin{equation}
	\begin{aligned}
	&C(k,k_1,k_2,k_3) := 1_{\{|k_{\max}-k|\le  10\}\cup \{k_{\max}\ge  k + 10,\ |k_{\max}-k_{\text{med}}|\le  10 \}}\\
	&2^{\max\{k,k_3\}-\max\{k_1,k_2\}}\min\{2^{-\frac{1}{2}(|k-k_1|+|k_2-k_3|)},2^{-\frac{1}{2}(|k-k_2|+|k_1-k_3|)}\} . 
	\end{aligned}
  \end{equation}
The indicator function restricts to the two main frequency interaction scenarios: either the highest frequency is close to $k$, or it is significantly larger than $k$ but close to the second highest.

To bound the sum in \eqref{DApsi-1}, we consider two subcases based on the relationship between $\max\{k,k_3\}$ and $\max\{k_1,k_2\}$.

If $\max\{k,k_3\}\le  \max\{k_1,k_2\}+10$, then 
\begin{align}
	&C(k,k_1,k_2,k_3)\lesssim \min\{2^{-\frac{1}{2}(|k-k_1|+|k_2-k_3|)},2^{-\frac{1}{2}(|k-k_2|+|k_1-k_3|)}\}\notag \\
	&\lesssim 2^{-\frac{1}{4}(|k-k_1|+|k_2-k_3|+|k-k_2|+|k_1-k_3|)}\notag\\
	&\lesssim 2^{-\frac{|\Delta k|}{8}},\ \Delta k:= \max\{k,k_1,k_2,k_3\}- \min\{k,k_1,k_2,k_3\}, 
\end{align}
  which provides strong off-diagonal decay.

  If $\max\{k,k_3\}\ge  \max\{k_1,k_2\}+10$, then the factor vanishes unless $k_{\max} = k_3$ and $|k_3-k|\le 10$. In this scenario we have $\max\{k_1,k_2\}\le  k +10$ and 
  \begin{align}
	&C(k,k_1,k_2,k_3)\lesssim 2^{k-\max\{k_1,k_2\}}\cdot 2^{-\frac{1}{2}(|k-k_1|+ |k-k_2|)}\\
	&\lesssim 2^{k-\max\{k_1,k_2\} }\cdot 2^{\frac{1}{2}(k_1+k_2-2k)}\lesssim 2^{-\frac{1}{2}|k_1-k_2|}.
  \end{align}
  The gain comes from the mismatch between $k$ and $\max\{k_1,k_2\}$.

Combining both subcases, we estimate 
\begin{align}
	\eqref{DApsi-1}&\lesssim 2^{-\sigma k }\mathcal{D}(P_k w) \sum_{\max\{k,k_3\}\le  \max\{k_1,k_2\}}2^{-\frac{|\Delta k|}{8}} b_{k_{\min}}b_{k_{\text{med}}}b_{k_{\max}}(\sigma) \notag\\
	&\quad +2^{-\sigma k }\mathcal{D}(P_k w) \sum_{k_1,k_2}2^{-\frac{1}{2}|k_1-k_2|}b_{k_1}b_{k_2}b_k(\sigma)  \notag\\
	&\lesssim \varepsilon 2^{-\sigma k }\mathcal{D}(P_k w)b_k(\sigma),
\end{align} 
which gives \eqref{N1} and thus completes the proof.
\end{proof}
\begin{proposition}
	Under bootstrap assumptions in Proposition \ref{SchrodingerBA}, we have 
	\begin{equation}
	  b_k(\sigma) +  \mathcal{D}(P_{k}\bm{\psi})\lesssim c_k(\sigma),\ \|P_k \psi_t\|_G \lesssim  2^k \varepsilon.
	\end{equation}
\end{proposition}
\begin{proof}
By the definition \eqref{bfre} of $b_k(\sigma)$, it suffices to prove the following estimates:
    \begin{align}
        &\mathcal{D}(P_k \bm{\psi}) \lesssim 2^{-\sigma k} c_k(\sigma),\ \sigma\in [0,\sigma_1]\label{goal1}\\
        &  \|P_k \psi_t\|_G \lesssim 2^{-(\sigma -1)k} c_k(\sigma),\ \sigma \in [\frac{1}{5},\sigma_1]. \label{goal2}\\
		&  \|P_k \psi_t\|_G \lesssim  2^k \varepsilon.\label{goal3}
    \end{align}

We first prove \eqref{goal1}. Applying the energy and Strichartz estimate \eqref{Stri} for the hyperbolic Schrödinger equation to the differentiated Ishimori equation \eqref{psi}, we obtain
 \begin{equation}
   \mathcal{D}(P_k \bm{\psi})\lesssim 2^{-\sigma k} c_k(\sigma) + \sum_{m=1,2}\|P_k\mathcal{N}_m\|_{DV^2_{\text{UH}}} + I(P_k \bm{\psi},P_k\mathcal{N}_m)^{\frac{1}{2}}.
 \end{equation}
The integral $I(P_k \bm{\psi},\mathcal{N}_m)^{\frac{1}{2}}$ is controlled directly by \eqref{N2} and \eqref{N1}. Using the bootstrap assumption $\mathcal{D}(P_k \bm{\psi}) \lesssim 2^{-\sigma k} b_k(\sigma)$ and $b_k(\sigma)\le  \varepsilon^{-\frac{1}{4}}c_k(\sigma)$, we obtain
\begin{equation}
  I(P_k \bm{\psi},P_k\mathcal{N}_m)^{\frac{1}{2}} \lesssim \mathcal{D}(P_k \bm{\psi})^{\frac{1}{2}}(\varepsilon2^{-\sigma k }b_k(\sigma))^{\frac{1}{2}}\lesssim \varepsilon^{\frac{1}{2}}2^{-\sigma k }b_k(\sigma)\lesssim 2^{-\sigma k }c_k(\sigma).
\end{equation}

The norm $\|P_k \mathcal{N}_m\|_{DV^2_{\text{UH}}}$ is estimated by duality. For any $U^2_{\text{UH}}$ atom $w$  (which implies $\mathcal{D}(P_k w) \lesssim 1$), we have
\begin{equation}
  \big|\int  \bar{w} P_k\  \mathcal{N}_m\mathrm{d}t\mathrm{d}x\big|\lesssim \|P_{\sim  k} w  P_k\mathcal{N}_m \|_{L^1_{t,x}}\lesssim \varepsilon 2^{-\sigma k }b_{k}(\sigma)\lesssim 2^{-\sigma k }c_k(\sigma). 
\end{equation}

It remains to prove \eqref{goal2} and \eqref{goal3} for $\psi_t$. Using the  equation \eqref{psit}, we have 
\begin{equation}
   \psi_t = \pa_x \bm{\psi} + \bm{A} \cdot \bm{\psi} + (R^2 \bm{A} )\cdot \bm{\psi}, 
\end{equation} 
where $R$ is the Riesz operators. The term $\partial_x \bm{\psi}$ is already controlled by the estimate for $\bm{\psi}$. The remaining terms involve products of the connection coefficient $\bm{A}$ and the field $\bm{\psi}$. To estimate these products, we use the following bound for $\bm{A}$, obtained by sending $s \to 0$ in the heat flow estimate \eqref{Ab} and using the previously established fact that $a_k(\sigma) \lesssim b_k(\sigma) \lesssim \varepsilon^{\frac{1}{2}} $:
    \begin{equation}\label{Aest}
        \|P_k \bm{A}\|_G \lesssim \varepsilon^{\frac{1}{2}} 2^{-\sigma k} b_k(\sigma).
    \end{equation}

Applying Lemma~\ref{Bony} to the product terms, combined with~\eqref{Aest} and the bootstrap bound for $\bm{\psi}$, we find for $\sigma \in [\tfrac15,\sigma_1]$ that 
    \begin{align}
        &\|P_k (\bm{A} \cdot \bm{\psi}, (R^2 \bm{A} )\cdot \bm{\psi})\|_G  \lesssim \sum_{k'\le k}  2^{k'-\sigma k}\varepsilon^{\frac{1}{2}}b_{k'}b_k(\sigma) +\sum_{k'\ge k} 2^{k-\sigma k'} \varepsilon^{\frac{1}{2}} b_{k'}b_{k'}(\sigma)\notag\\
		&\lesssim 2^{-(\sigma -1) k }c_k(\sigma)\Big(1 + \sum_{k'\ge k}2^{\sigma(k-k')}2^{\delta(k'-k)}\Big)\lesssim 2^{-(\sigma -1) k }c_k(\sigma).
    \end{align} 
Similarly,
    \begin{align}
        &\|P_k (\bm{A} \cdot \bm{\psi}, (R^2 \bm{A} )\cdot \bm{\psi})\|_G  \lesssim \sum_{k'\le k}  2^{k' }\varepsilon^{\frac{1}{2}}b_{k'}b_k  +\sum_{k'\ge k} 2^{k } \varepsilon^{\frac{1}{2}} b_{k'}^2\lesssim 2^k \varepsilon.
    \end{align} 
These estimates control the remaining terms and complete the proof.
\end{proof}
\makeatletter
\nocite{*}
{\renewcommand{\addcontentsline}[3]{}
\section*{Acknowledgements}} 
\makeatother
  Y. Zhou was supported by NSFC
(No.\,12571231, No.\,12171097). 
\vspace{0.6cm}
  \bibliographystyle{plain}
 \bibliography{ref}

\begin{thebibliography}{10}

\bibitem{bejenaru_stability_2011}
I.~Bejenaru, A.~D. Ionescu, and C.~E. Kenig.
\newblock On the {{Stability}} of {{Certain Spin Models}} in 2+1 {{Dimensions}}.
\newblock {\em J Geom Anal}, 21(1):1--39, January 2011.

\bibitem{bejenaru_global_2011}
I.~Bejenaru, A.~D. Ionescu, C.~E. Kenig, and D.~Tataru.
\newblock Global {{Schr\"odinger}} maps in dimensions d{$\geq$}2: {{Small}} data in the critical {{Sobolev}} spaces.
\newblock {\em Ann. Math.}, 173(3):1443--1506, May 2011.

\bibitem{2015A}
Benjamin Dodson and Paul Smith.
\newblock A controlling norm for energy-critical schr"odinger maps.
\newblock {\em Transactions of the American Mathematical Society}, 367, 2015.

\bibitem{ifrim_global_2025}
Mihaela Ifrim, Ben Pineau, and Daniel Tataru.
\newblock Global solutions for cubic quasilinear ultrahyperbolic {{Schr\"odinger}} flows, April 2025.

\bibitem{ionescu_lowregularity_2006}
Alexandru~D. Ionescu and Carlos~E. Kenig.
\newblock Low-regularity {{Schr\"odinger}} maps.
\newblock {\em Differential and Integral Equations}, 19(11):1271--1300, January 2006.

\bibitem{ionescu_lowregularity_2007}
Alexandru~D. Ionescu and Carlos~E. Kenig.
\newblock Low-regularity {{Schr\"odinger}} maps, {{II}} : Global well-posedness in dimensions d $\geq$ 3.
\newblock {\em Commun. Math. Phys.}, 271(2):523--559, February 2007.

\bibitem{ishimori_multivortex_1984}
Yuji Ishimori.
\newblock Multi-vortex solutions of a two-dimensional nonlinear wave equation.
\newblock {\em Prog Theor Phys}, 72(1):33--37, July 1984.

\bibitem{kenig_Cauchy_2005}
Carlos~E Kenig and Andrea~R Nahmod.
\newblock The {{Cauchy}} problem for the hyperbolic--elliptic {{Ishimori}} system and {{Schr\"odinger}} maps.
\newblock {\em Nonlinearity}, 18(5):1987--2009, September 2005.

\bibitem{koch_Dispersive_2014}
Herbert Koch, Daniel Tataru, and Monica Vi{\c s}an.
\newblock {\em Dispersive {{Equations}} and {{Nonlinear Waves}}: {{Generalized Korteweg}}--de {{Vries}}, {{Nonlinear Schr\"odinger}}, {{Wave}} and {{Schr\"odinger Maps}}}, volume~45 of {\em Oberwolfach {{Seminars}}}.
\newblock Springer Basel, Basel, 2014.

\bibitem{linares_Introduction_2015}
Felipe Linares and Gustavo Ponce.
\newblock {\em Introduction to Nonlinear Dispersive Equations}.
\newblock Universitext. Springer New York, New York, NY, 2015.

\bibitem{soyeur_Cauchy_1992}
Alain Soyeur.
\newblock The {{Cauchy}} problem for the {{Ishimori}} equations.
\newblock {\em Journal of Functional Analysis}, 105(2):233--255, May 1992.

\bibitem{tao_Global_2001}
Terence Tao.
\newblock Global regularity of wave maps {{II}}. small energy in two dimensions.
\newblock {\em Communications in Mathematical Physics}, 224(2):443--544, December 2001.

\bibitem{tataru_Rough_2005}
Daniel Tataru.
\newblock Rough solutions for the wave maps equation.
\newblock {\em American Journal of Mathematics}, 127(2):293--377, 2005.

\bibitem{wang_periodic_2024}
Sheng Wang and Yi~Zhou.
\newblock Periodic {{Schr\"odinger}} map flow on {{K\"ahler}} manifolds, February 2024.

\bibitem{wang_Local_2012}
Yuzhao Wang.
\newblock Local well-posedness for hyperbolic--elliptic {{Ishimori}} equation.
\newblock {\em Journal of Differential Equations}, 252(9):4625--4655, May 2012.

\bibitem{zhou_1+2dimensional_2022}
Yi~Zhou.
\newblock (1+2)-{{Dimensional Radially Symmetric Wave Maps Revisit}}.
\newblock {\em Chin. Ann. Math. Ser. B}, 43(5):785--796, September 2022.

\end{thebibliography}

\end{document}